\documentclass[12pt]{article}
\usepackage{amsmath}
\usepackage{amssymb}
\usepackage{textcomp}
\usepackage{amsmath,amsthm,amsfonts,amssymb,amstext}
\usepackage{layout}
\usepackage{diagrams}
\newcommand{\K}{\ensuremath{\mathbb{K}}}

\newtheorem{thm}{Theorem}[section]
\newtheorem{cor}[thm]{Corollary}

\newtheorem{lemma}[thm]{Lemma}

\newtheorem{note}[thm]{Note}
\newtheorem{definition}[thm]{Definition}

\newtheorem{prob}[thm]{Problem}
\renewenvironment{proof}{\medskip\noindent{\emph {Proof:}\ }}{\qed\medskip}

\begin{document}

\title{ \bf Two commuting operators associated\\ with a tridiagonal pair}
\author{Sarah Bockting-Conrad
%\footnote{
% Mathematics  Department,  
%University of  Wisconsin, 
% 480 Lincoln Drive, 
%Madison, WI 53706 \ \
%Email:bockting@math.wisc.edu }
}
\date{}
%to get date printout, comment out above line 
\maketitle

\begin{abstract}
Let $\K$ denote a field and let $V$ denote a vector space over $\K$ with finite positive dimension.  We consider an ordered pair of linear transformations $A: V\to V$ and $A^*: V \to V$ that satisfy the following four conditions:  
(i) Each of $A,A^*$ is diagonalizable;
(ii) there exists an ordering $\{ V_i\}_{i=0}^d$ of the eigenspaces of $A$ such that 
$A^* V_i \subseteq V_{i-1}+V_i+V_{i+1}
$ for $0\leq i\leq d$, where $V_{-1}=0$ and $V_{d+1}=0$;
(iii) there exists an ordering $\{V^*_i\}_{i=0}^{\delta}$ of the eigenspaces of $A^*$ such that 
$A V^*_i \subseteq V^*_{i-1}+V^*_i+V^*_{i+1}
$ for $0\leq i\leq \delta$, where $V^*_{-1}=0$ and $V^*_{\delta+1}=0$;
(iv) there does not exist a subspace $W$ of $V$ such that $AW \subseteq W$, $A^*W\subseteq W$, $W\neq 0$, $W\neq V$.  
We call such a pair a {\it tridiagonal pair} on $V$.  
It is known that $d=\delta$; to avoid trivialities assume $d\geq 1$.  
We show that there exists a unique linear transformation $\Delta:V\to V$ such that
$(\Delta -I)V^*_i\subseteq V_0^{*}+V_1^{*}+\cdots +V_{i-1}^{*}$
and
$\Delta (V_i+V_{i+1}+\cdots+V_d )= V_0 +V_{1}+\cdots+V_{d-i}$
for $0\leq i \leq d$.
We show that there exists a unique linear transformation $\Psi:V\to V$ such that 
$\Psi V_i\subseteq V_{i-1}+V_i+V_{i+1}$ 
and 
$\left(\Psi-\Lambda \right)V_i^*\subseteq V_0^*+V_1^*+\cdots + V_{i-2}^*$  for $0\leq i\leq d$,
where $\Lambda=(\Delta-I)(\theta_0-\theta_d)^{-1}$ and $\theta_0$ (resp. $\theta_d$) denotes the eigenvalue of $A$ associated with $V_0$ (resp. $V_d$).
 We characterize $\Delta,\Psi$ in several ways.  
There are two well-known decompositions of 
$V$ called the first and second split decomposition.  We discuss how $\Delta,\Psi$ act on these decompositions.
We also show how $\Delta,\Psi$ relate to each other.
Along this line we have two main results.
Our first main result is that $\Delta, \Psi$ commute.  
In the literature on TD pairs, there is a scalar $\beta$ used to describe the eigenvalues.  
Our second main result is that each of $\Delta^{\pm 1}$ is a polynomial of degree $d$ in $\Psi$, under a minor assumption on $\beta$.

\end{abstract}

\section{Introduction}\label{section:intro}

Throughout this paper, $\K$ denotes a field and $\overline{\K}$ denotes the algebraic closure of $\K$.\\

We begin by recalling the notion of a tridiagonal pair.
We will use the following terms.  Let $V$ denote a vector space over $\K$ with finite positive dimension.  For a linear transformation $A:V\to V$ and a subspace $W\subseteq V$, we say that $W$ is an {\it eigenspace} of $A$ whenever $W\neq 0$ and there exists $\theta\in\K$ such that $W=\{ v\in V | Av=\theta v\}$.  In this case, $\theta$ is called the {\it eigenvalue} of $A$ associated with $W$.  We say that $A$ is {\it diagonalizable} whenever $V$ is spanned by the eigenspaces of $A$.

\begin{definition}  {\rm \cite[Definition 1.1]{Somealg}} \label{def:tdp}
{\rm Let $V$ denote a vector space over $\K$ with finite positive dimension. 
By a {\em tridiagonal pair} (or {\em TD pair}) on $V$ we mean an ordered pair of linear 
transformations $A:V \to V$ and $A^*:V \to V$ that satisfy the following 
four conditions.
\begin{enumerate}
\item[{\rm (i)}] 
Each of $A,A^*$ is diagonalizable.
\item[{\rm (ii)}] 
There exists an ordering $\{V_i\}_{i=0}^d$ of the eigenspaces of $A$ 
such that 
\begin{equation}               \label{eq:t1}
A^* V_i \subseteq V_{i-1} + V_i+ V_{i+1} \qquad \qquad (0 \leq i \leq d),
\end{equation}
where $V_{-1} = 0$ and $V_{d+1}= 0$.
\item[{\rm (iii)}]
There exists an ordering $\{V^*_i\}_{i=0}^{\delta}$ of the eigenspaces of 
$A^*$ such that 
\begin{equation}                \label{eq:t2}
A V^*_i \subseteq V^*_{i-1} + V^*_i+ V^*_{i+1} 
\qquad \qquad (0 \leq i \leq \delta),
\end{equation}
where $V^*_{-1} = 0$ and $V^*_{\delta+1}= 0$.
\item [{\rm (iv)}]
There does not exist a subspace $W$ of $V$ such  that $AW\subseteq W$,
$A^*W\subseteq W$, $W\not=0$, $W\not=V$.
\end{enumerate}
We say the pair $A,A^*$ is {\em over $\K$}.}
\end{definition}
\medskip

\begin{note}   \label{note:star}        \samepage
{\rm According to a common notational convention $A^*$ denotes 
the conjugate-transpose of $A$. We are not using this convention.
In a TD pair $A,A^*$ the linear transformations $A$ and $A^*$
are arbitrary subject to (i)--(iv) above.}
\end{note}

\medskip

Referring to the TD pair in Definition \ref{def:tdp}, by  \cite[Lemma 4.5]{Somealg} 
the scalars 
 $d$ and $\delta $ are equal.
We call this common value the {\it diameter} of $A,A^*$.
To avoid trivialities, throughout this paper we assume that the diameter is at least one.\\

We now give some background on TD pairs.  The concept of a TD pair originated in the theory of $Q$-polynomial distance-regular graphs \cite{subconstituent}.  Since that beginning the TD pairs have been investigated in a systematic way; for notable papers along this line see 
\cite{TDclassification, Somealg, DP, shape, td-uqsl2, cubicqserre, qtetalgebra, augTDalg, IT:qRacah,  sharpen, LPintro, 2LT}.
Several of these papers focus on a class of TD pair said to be sharp.  These papers ultimately led to the classification of sharp TD pairs \cite{TDclassification}.  In spite of this classification, there are still some intriguing aspects of TD pairs which have not yet been fully studied.  In this paper, we investigate one of those aspects.\\

We now summarize the present paper.  
Given a TD pair $A,A^*$ on $V$ we introduce 
two linear transformations $\Delta:V\to V$ and $\Psi:V\to V$ that we find attractive.
We characterize $\Delta,\Psi$ in several ways.  
There are two well-known decompositions of 
$V$ called the first and second split decomposition \cite[Section 4]{Somealg}.  We discuss how $\Delta,\Psi$ act on these decompositions.
We also show how $\Delta,\Psi$ relate to each other.\\

We now describe $\Delta,\Psi$ in more detail.  
For the rest of this section, fix an ordering $\{V_i\}_{i=0}^d$ (resp. $\{V_i^*\}_{i=0}^d$) of the eigenspaces of $A$ (resp. $A^*$) which satisfies (\ref{eq:t1}) (resp. (\ref{eq:t2})).  For $0\leq i\leq d$ let $\theta_i$ (resp. $\theta_i^*$) denote the eigenvalue of $A$ (resp. $A^*$) corresponding to $V_i$ (resp. $V_i^*$).  
We show that there exists a unique linear transformation $\Delta:V\to V$ such that both 
\begin{align*}
&(\Delta -I)V^*_i\subseteq V_0^{*}+V_1^{*}+\cdots +V_{i-1}^{*},\\
&\Delta (V_i+V_{i+1}+\cdots+V_d )= V_0 +V_{1}+\cdots+V_{d-i}
\end{align*}
for $0\leq i \leq d$.
We show that there exists a unique linear transformation $\Psi:V\to V$ such that both 
\begin{align*}
&\Psi V_i\subseteq V_{i-1}+V_i+V_{i+1},\\
&\left(\Psi-\frac{\Delta-I}{\theta_0-\theta_d}\right)V_i^*\subseteq V_0^*+V_1^*+\cdots + V_{i-2}^*
\end{align*}
for $0\leq i\leq d$.  
By construction, 
\begin{equation*}
\Psi V^*_i\subseteq V_0^{*}+V_1^{*}+\cdots +V_{i-1}^{*} \qquad\qquad (0\leq i\leq d).
\end{equation*}
\medskip

Before discussing $\Delta$ and $\Psi$ further, 
we recall the split decompositions of $V$.  
For $0\leq i\leq d$ define
\begin{align*}
U_i&= (V^*_0+V^*_1+\cdots + V^*_i)\cap (V_i+V_{i+1}+\cdots + V_d),\\
U_i^{\Downarrow} &= (V^*_0+V^*_1+\cdots + V^*_i)\cap (V_0+V_1+\cdots + V_{d-i}).
\end{align*}
By \cite[Theorem 4.6]{Somealg},
both the sums $V=\sum_{i=0}^d U_i$ and $V=\sum_{i=0}^d U_i^{\Downarrow}$ are direct.  We call $\{U_i\}_{i=0}^d$ (resp. $\{U_i^{\Downarrow}\}_{i=0}^d$) the {\it first split decomposition} (resp. {\it second split decomposition}) of $V$.
By \cite[Theorem 4.6]{Somealg}, the maps $A,A^*$ act on the split decompositions in the following way:
\begin{align*}
&(A-\theta_i I)U_i\subseteq U_{i+1} &(0\leq i \leq d-1 ), \qquad &(A-\theta_d I)U_d=0,\\
&(A^*-\theta_i^* I)U_i\subseteq U_{i-1} &(1\leq i\leq d ), \qquad &(A^*-\theta_0^* I)U_0=0
\end{align*}
and
\begin{align*}&(A-\theta_{d-i} I)U_i^{\Downarrow}\subseteq U_{i+1}^{\Downarrow}&(0\leq i\leq d-1 ), \qquad &(A-\theta_0 I)U_d^{\Downarrow}=0,\\
&(A^*-\theta_i^* I)U_i^{\Downarrow}\subseteq U_{i-1}^{\Downarrow}&(1\leq i\leq d ), \qquad &(A^*-\theta_0^* I)U_0^{\Downarrow}=0.
\end{align*}
\medskip

We now summarize how $\Delta,\Psi$ act on the split decompositions of $V$.
We show that
\begin{align*}
&\Delta U_i=U_i^{\Downarrow},\\
&(\Delta-I)U_i\subseteq U_0+U_1+\cdots + U_{i-1},\\
&(\Delta-I)U_i^{\Downarrow}\subseteq U_0^{\Downarrow}+U_1^{\Downarrow}+\cdots + U_{i-1}^{\Downarrow}
\end{align*}
for $0\leq i\leq d$.
We also show that 
\begin{align*}
&\Psi U_i\subseteq U_{i-1} &(1\leq i\leq d), \qquad & \Psi U_0=0,\\
&\Psi U_i^{\Downarrow} \subseteq U_{i-1}^{\Downarrow} &(1\leq i\leq d), \qquad & \Psi U_0^{\Downarrow}=0.
\end{align*}
\medskip

We now discuss how $\Delta,\Psi$ relate to each other.  Along this line we have two main results.
Our first main result is that $\Delta, \Psi$ commute.  In order to state the second result, we define
\begin{align*}
\vartheta_i=\sum_{h=0}^{i-1}\frac{\theta_h-\theta_{d-h}}{\theta_0-\theta_d}\qquad\qquad (1\leq i\leq d).
\end{align*}
Our second main result is that both
\begin{eqnarray*}
\Delta = I +\frac{\eta_{1}(\theta_0)}{\vartheta_1}\Psi + \frac{\eta_{2}(\theta_0)}{\vartheta_1\vartheta_2}\Psi^2 +\cdots + \frac{\eta_{d}(\theta_0)}{\vartheta_1\vartheta_2 \cdots \vartheta_d} \Psi^d,\\
\Delta^{-1} = I +\frac{\tau_{1}(\theta_d)}{\vartheta_1}\Psi + \frac{\tau_{2}(\theta_d)}{\vartheta_1\vartheta_2}\Psi^2  +\cdots+ \frac{\tau_{d}(\theta_d)}{\vartheta_1\vartheta_2 \cdots \vartheta_d} \Psi^d
\end{eqnarray*}
provided that each of $\vartheta_1,\vartheta_2,\mathellipsis,\vartheta_d$ is nonzero.
Here $\tau_i, \eta_i$ are the polynomials
\begin{align*} 
\tau_{i}&=(x-\theta_0)(x-\theta_{1})\cdot\cdot\cdot (x-\theta_{i-1}),\\
\eta_{i}&=(x-\theta_{d})(x-\theta_{d-1})\cdot\cdot\cdot (x-\theta_{d-i+1})
\end{align*}
for $0\leq i\leq d$.
In the literature on TD pairs there is a scalar $\beta$ that is used to describe the eigenvalues of $A$ and $A^*$ \cite[Sections 10 and 11]{Somealg}. 
We show that each of $\vartheta_1,\vartheta_2,\ldots,\vartheta_d$ is nonzero if and only if neither of the following holds: (i) $\beta=-2$, $d$ is odd, and ${\rm Char}(\K)\neq 2$; (ii) $\beta=0$, $d=3$, and ${\rm Char}(\K)=2$.
We conclude the paper with a few comments on further research.\\

\section{Preliminaries}\label{section:prelim}

When working with a tridiagonal pair, it is useful to consider a closely related object called a tridiagonal system. In order to define this, we first recall some facts from elementary linear algebra.\\

Let $V$ denote a vector space over $\K$ with finite positive dimension.
Let $ {\rm End} (V)$ denote the $\K$-algebra consisting of all linear transformations from $V$ to $V$.
Let $A$ denote a diagonalizable element in ${\rm End}(V)$.  
Let $\{V_i\}_{i=0}^d$ denote an ordering of the eigenspaces of $A$.
For $0\leq i\leq d$ let $\theta_i$ be the eigenvalue of $A$ corresponding to $V_i$.
Define $E_i\in {\rm End}(V)$ by
\begin{align}
(E_i-I)V_i&=0,\\
E_iV_j=0 \quad \text{  if  }&\quad j\neq i,\qquad (0\leq j\leq d).
\end{align}
In other words, $E_i$ is the projection map from $V$ onto $V_i$.  
We refer to $E_i$ as the {\it primitive idempotent} of
$A$ associated with $\theta_i$.
By elementary linear algebra, we have
\begin{align}
AE_i = E_iA = \theta_iE_i \qquad  (0 \leq i \leq d),
\label{eq:primid1}
\\
\quad E_iE_j = \delta_{ij}E_i\qquad (0 \leq i,j\leq d),
\label{eq:primid2}
\\
\quad V_i=E_iV  \quad \qquad  (0 \leq i \leq d),
\label{eq:primid0}
\\
I=\sum_{i=0}^d E_i.\qquad\qquad
\label{eq:primid3}
\end{align}
One readily checks that 
\begin{eqnarray*}
E_i = \prod_{{0 \leq  j \leq d}\atop{j\not=i}} {{A-\theta_j I}\over {\theta_i-\theta_j}}\qquad \qquad (0 \leq i \leq d).\label{EA} 
\end{eqnarray*}

Let $M$ denote the $\K$-subalgebra of ${\rm End}(V)$ generated by $A$.  
We note that each of $\{ A^i\}_{i=0}^d$, $\{E_i\}_{i=0}^d$ is a basis for the $\K$-vector space $M$.\\  

Given a TD pair $A,A^*$ on $V,$ an ordering of the eigenspaces of $A$ (resp. $A^*$) is said to be {\it standard} whenever (\ref{eq:t1}) (resp. (\ref{eq:t2})) holds.  
Let $\{V_i\}_{i=0}^d$ denote a standard ordering of the eigenspaces of $A$.  By \cite[Lemma 2.4]{Somealg}, the ordering $\{V_{d-i}\}_{i=0}^d$ is standard and no further ordering is standard.  A similar result holds for the eigenspaces of $A^*$.  An ordering of the primitive idempotents of $A$ (resp. $A^*$) is said to be {\it standard} whenever the corresponding ordering of the eigenspaces of $A$ (resp. $A^*$) is standard.

\begin{definition}{\rm \cite[Definition 2.1]{TDclass}} \label{def:TDsys} {\rm
Let $V$ denote a vector space over $\K$ with finite positive dimension.  
By a {\it tridiagonal system} (or {\em TD system}) on $V,$  we mean a 
sequence 
\begin{equation*}
\Phi = (A; \{ E_i\}_{i=0}^d; A^*;\{ E_i^*\}_{i=0}^d)
\end{equation*}
 that satisfies  (i)--(iii) below.
\begin{enumerate}
\item[{\rm (i)}] $A, A^*$ is a tridiagonal pair on $V$. 
\item [{\rm (ii)}]$\{ E_i\}_{i=0}^d$ is a standard ordering of the primitive 
idempotents of $\;A$.
\item[{\rm (iii)}] $\{ E_i^*\}_{i=0}^d$ is a standard ordering of the primitive 
idempotents of $\;A^*$.
\end{enumerate}
We call $d$ the {\it diameter} of $\Phi$,  
and say $\Phi$ is {\it over } $\K$. 
 For notational convenience, set $E_{-1}=0$, $E_{d+1}=0$, $
E^*_{-1}=0$, $E^*_{d+1}=0$.}\\
\end{definition}

In Definition \ref{def:TDsys} we do not assume that the primitive idempotents $\{ E_i\}_{i=0}^d,\{ E_i^*\}_{i=0}^d$ all 
have rank 1.  A TD system for which each these primitive idempotents does have rank 1 is called a Leonard system \cite{2LT}.
The Leonard systems  are classified up to isomorphism \cite[Theorem 1.9]{2LT}.\\

%%%%%%%%%%%%%%%%%%%%%%%%%%%%%%%%%%%%%%%%%%%%%%%%%%

For the rest of the present paper, we fix a TD system $\Phi$ as in Definition \ref{def:TDsys}.\\

\begin{definition}\label{def:main}{\rm
For $0\leq i\leq d$ let $\theta_i$ (resp. $\theta_i^*$) denote the eigenvalue of $A$ (resp. $A^*$) associated with $E_i$ (resp. $E_i^*$).
We refer to $\{\theta_i\}_{i=0}^d$ (resp. $\{\theta_i^*\}_{i=0}^d$) as the {\it eigenvalue sequence} (resp. {\it dual eigenvalue sequence}) of $\Phi$. 
}\end{definition}

%%%%%%%%%%%%%%%%%%%%%%%%%%%%%%%%%%%%%%%%%%%%%%%%%%

A given TD system can be modified in a number of ways to get a new TD system.  For example, given the TD system $\Phi$ in Definition \ref{def:TDsys},
the sequence
\begin{eqnarray*}
\Phi^{\Downarrow} = (A;\{ E_{d-i}\}_{i=0}^d;A^*;\{ E_i^*\}_{i=0}^d)
\end{eqnarray*}
is a TD system on $V$. 
Following  {\rm \cite[Section 3]{Somealg}}, we call $\Phi^{\Downarrow}$ the {\it second inversion} of $\Phi$.
When discussing $\Phi^{\Downarrow}$, we use the following notational convention.
For any object $f$ associated with $\Phi$ we let $f^{\Downarrow}$ denote the corresponding object for $\Phi^{\Downarrow}$.\\  

%%%%%%%%%%%%%%%%%%%%%%%%%%%%%%%%%%%%%%%%%%%%%%%%%%

For later use, we associate with $\Phi$ two families of polynomials as follows.
Let $x$ be an indeterminate.  Let $\K [x]$ denote the $\K$-algebra consisting of the polynomials in $x$ that have all coefficients in $\K$.  For $0\leq i\leq j\leq d+1$, we define the polynomials $\tau_{ij}=\tau_{ij}(\Phi)$, $\eta_{ij}=\eta_{ij}(\Phi)$ in $\K [x]$ by
\begin{align} \tau_{ij}&=(x-\theta_i)(x-\theta_{i+1})\cdot\cdot\cdot (x-\theta_{j-1}), \label{tau}\\
\eta_{ij}&=(x-\theta_{d-i})(x-\theta_{d-i-1})\cdot\cdot\cdot (x-\theta_{d-j+1}). \label{eta}\end{align}
We interpret $\tau_{i,i-1}=0$ and $\eta_{i,i-1}=0$.  
Note that each of $\tau_{ij}$, $\eta_{ij}$ is monic with degree $j-i$.  In particular, $\tau_{ii}=1$ and $\eta_{ii}=1$.   
We remark that 
$\tau_{ij}^{\Downarrow}=\eta_{ij}$ and $\eta_{ij}^{\Downarrow}=\tau_{ij}$.\\

Observe that for $0\leq i\leq j\leq k\leq d+1$, 
\begin{equation}
\tau_{ij} \tau_{jk} = \tau_{ik}, \hspace{1 in}
\eta_{ij} \eta_{jk} = \eta_{ik}.\label{eq:taumult}
\end{equation}
  
As we proceed through the paper, we will focus on $\tau_{ij}$.  We will develop a number of results concerning $\tau_{ij}$.  Similar results hold for $\eta_{ij}$, although we will not state them explicitly.

\begin{lemma}\label{lemma:taukernel}
For $0\leq i\leq j\leq d+1$,
the kernel of $\tau_{ij}(A)$ is 
\begin{equation*}
E_iV+E_{i+1}V+\cdots +E_{j-1}V.
\end{equation*}
\end{lemma}
\begin{proof}
For $0\leq h\leq d$, $E_hV$ is the eigenspace of $A$ corresponding to $\theta_h$.  The result follows from this and (\ref{tau}).
\end{proof}

\medskip

For $0\leq j \leq d+1$, we abbreviate 
\begin{equation*}
\tau_j=\tau_{0j},\qquad \qquad \eta_j=\eta_{0j}.
\end{equation*} Thus
\begin{align} 
\tau_{j}&=(x-\theta_0)(x-\theta_{1})\cdot\cdot\cdot (x-\theta_{j-1}),\label{eq:tau_j}\\
\eta_{j}&=(x-\theta_{d})(x-\theta_{d-1})\cdot\cdot\cdot (x-\theta_{d-j+1}).\label{eq:eta_j}
\end{align}

\bigskip

In our discussion of $\Psi$, the following scalars will be useful.

\begin{definition}{\rm \cite[Section 10]{2LT}}\label{def:vartheta}
{\rm For $0\leq i \leq d+1$, define
\begin{equation*}
\vartheta_i = \sum_{h=0}^{i-1}\frac{\theta_h-\theta_{d-h}}{\theta_0-\theta_d}.
\end{equation*}
}
\end{definition}

We observe that
\begin{eqnarray}
\vartheta_{i+1}-\vartheta_i=\frac{\theta_i-\theta_{d-i}}{\theta_0-\theta_d} \qquad\qquad (0\leq i\leq d).\label{line:varthetadiff}
\end{eqnarray}

These scalars will be discussed further in Section \ref{section:vartheta}.

%%%%%%%%%%%%%%%%%%%%%%%%%%%%%%%%%%%%%%%%%%%%%%%%%%
%%%%%%%%%%%%%%%%%%%%%%%%%%%%%%%%%%%%%%%%%%%%%%%%%%
%%%%%%%%%%%%%%%%%%%%%%%%%%%%%%%%%%%%%%%%%%%%%%%%%%
\section{The first split decomposition of $V$}\label{section:U}
%%%%%%%%%%%%%%%%%%%%%%%%%%%%%%%%%%%%%%%%%%%%%%%%%%
%%%%%%%%%%%%%%%%%%%%%%%%%%%%%%%%%%%%%%%%%%%%%%%%%%
%%%%%%%%%%%%%%%%%%%%%%%%%%%%%%%%%%%%%%%%%%%%%%%%%%
We continue to discuss the TD system $\Phi$ from Definition \ref{def:TDsys}.\\  

We use the following concept.
By a {\it decomposition} of $V,$ we mean a sequence of subspaces whose direct sum is $V$.  
For example, $\{ E_iV\}_{i=0}^d$ and $\{ E_i^*V\}_{i=0}^d$ are decompositions of $V$.  
There are two more decompositions of $V$ of interest called the first and second split decomposition.  In this section, we discuss the first split decomposition of $V$.  In Section  \ref{section:Udd}, we will discuss the second split decomposition of $V$.

\begin{definition}\label{def:U}{\rm
For $0\leq i\leq d$ define 
\begin{equation*}
U_i = (E^*_0V+E^*_1V+\cdots + E^*_iV)\cap (E_iV+E_{i+1}V+\cdots + E_dV).
\end{equation*}
For notational convenience, define $U_{-1}=0$ and $U_{d+1}=0$.}
\end{definition}

%%%%%%%%%%%%%%%%%%%%%%%%%%%%%%%%%%%%%%%%%%%%%%%%%%

\begin{thm}{\rm \cite[Theorem 4.6]{Somealg} }\label{thm:U}
The sequence $\{U_i\}_{i=0}^d$ is a decomposition of $V$.  
Moreover, the following {\rm (i)--(iii)} hold.
\begin{enumerate}
\item[{\rm (i)}] $(A-\theta_i I)U_i\subseteq U_{i+1}\qquad (0\leq i \leq d-1 ), \qquad (A-\theta_d I)U_d=0$.
\item[{\rm (ii)}] $(A^*-\theta_i^* I)U_i\subseteq U_{i-1}\qquad (1\leq i\leq d ), \qquad (A^*-\theta_0^* I)U_0=0$.
\item[{\rm (iii)}] For $0\leq i \leq d$ both
\begin{align*}
U_i+U_{i+1}+\cdots + U_d&=E_iV+E_{i+1}V+\cdots + E_dV,\\
U_0+U_{1}+\cdots + U_i&=E_0^*V+E_{1}^*V+\cdots + E_i^*V.
\end{align*}
\end{enumerate}
\end{thm}

\begin{definition}
{\rm With reference to Definition \ref{def:U}, 
we refer to the sequence $\{U_i\}_{i=0}^d$ as the {\it first split decomposition} of $V$.
}\end{definition}

%%%%%%%%%%%%%%%%%%%%%%%%%%%%%%%%%%%%%%%%%%%%%%%%%%

\begin{lemma}{\rm \cite[Corollary 5.7]{Somealg}} \label{lemma:dim1}
For $0\leq i\leq d$ the dimensions of
$E_iV$, $E_i^*V$, $U_i$ coincide.  Denoting this common dimension by $\rho_i$, we have $\rho_i=\rho_{d-i}$.
\end{lemma}

%%%%%%%%%%%%%%%%%%%%%%%%%%%%%%%%%%%%%%%%%%%%%%%%%%

\begin{definition}\label{def:shape}{\rm \cite[Section 1]{shape}
With reference to Lemma \ref{lemma:dim1}, we refer to the sequence $\{\rho_i\}_{i=0}^d$ as the {\it shape} of $\Phi$.  Note that $\Phi$ and $\Phi^{\Downarrow}$ have the same shape.}
\end{definition}

%%%%%%%%%%%%%%%%%%%%%%%%%%%%%%%%%%%%%%%%%%%%%%%%%%
\begin{lemma}\label{lemma:AU}
Both\begin{align}
&AU_i\subseteq U_i+U_{i+1} \qquad (0\leq i\leq d-1), \qquad   AU_d\subseteq U_d,\label{eq:AU1}\\
&A^*U_i\subseteq U_i+U_{i-1} \qquad \quad (1\leq i\leq d), \qquad A^*U_0\subseteq U_0.\nonumber
\end{align}
\end{lemma}
\begin{proof}
Use Theorem \ref{thm:U}(i),(ii).
\end{proof}

%%%%%%%%%%%%%%%%%%%%%%%%%%%%%%%%%%%%%%%%%%%%%%%%%%
\begin{cor}\label{cor:A^kU}
For $0\leq i\leq d$ both
\begin{eqnarray}
A^k U_i\subseteq U_i + U_{i+1}+\cdots + U_{i+k} & \qquad (0\leq k\leq d-i), \label{eq:A^kU} \\
\left(A^*\right)^k U_i\subseteq U_i + U_{i-1}+\cdots + U_{i-k} & \qquad (0\leq k\leq i).\label{eq:A^kU2}
\end{eqnarray}  
\end{cor}
\begin{proof}
Use Lemma \ref{lemma:AU}.
\end{proof}
%%%%%%%%%%%%%%%%%%%%%%%%%%%%%%%%%%%%%%%%%%%%%%%%%%

\begin{definition}{\rm \cite[Definition 5.2]{Somealg}} \label{def:F_i}{\rm
For $0\leq i\leq d$ define $F_i\in {\rm End}(V)$ by 
\begin{align}
(F_i-I)U_i&=0,\label{eq:F_i1}\\
F_iU_j=0\quad \text{  if  }&\quad j\neq i,\qquad (0\leq j\leq d).\label{eq:F_i2}
\end{align}
In other words, $F_i$ is the projection map from $V$ onto $U_i$.  For notational convenience, define $F_{-1}=0$ and $F_{d+1}=0$.}
\end{definition}

%%%%%%%%%%%%%%%%%%%%%%%%%%%%%%%%%%%%%%%%%%%%%%%%%%

\begin{lemma}{\rm \cite[Lemma 5.3]{Somealg}} \label{lemma:F}
With reference to Definition \ref{def:F_i}, both
\begin{eqnarray}
F_iF_j=\delta_{ij}F_i &\qquad (0\leq i,j\leq d),\label{eq:FF}\\
I=\sum_{i=0}^d F_i.&\nonumber
\end{eqnarray}
\end{lemma}

%%%%%%%%%%%%%%%%%%%%%%%%%%%%%%%%%%%%%%%%%%%%%%%%%%

\begin{definition}{\rm \cite[Definition 6.1]{Somealg}} \label{def:R}{\rm
Define
\begin{equation*}
R = A-\sum_{h=0}^{d}\theta_h F_h,\qquad \qquad
L = A^*-\sum_{h=0}^{d}\theta_h^* F_h.\\
\end{equation*}
We refer to $R$ (resp. $L$) as the {\it raising map} (resp. {\it lowering map}) for $\Phi$.} 
\end{definition}

%%%%%%%%%%%%%%%%%%%%%%%%%%%%%%%%%%%%%%%%%%%%%%%%%%

\begin{lemma}{\rm \cite[Lemma 6.2]{Somealg}} \label{lemma:RA}
For $0\leq i\leq d$
the following hold on $U_i$. 
\begin{eqnarray*}
R=A-\theta_i I,\qquad \qquad L=A^*-\theta_i^* I.
\end{eqnarray*}
\end{lemma}

%%%%%%%%%%%%%%%%%%%%%%%%%%%%%%%%%%%%%%%%%%%%%%%%%%
\bigskip
Combining Theorem \ref{thm:U}(i),(ii) with Lemma \ref{lemma:RA} we obtain the following result.

\medskip

\begin{lemma}\label{lemma:RULU} Both
\begin{eqnarray}
&RU_i\subseteq U_{i+1} \qquad (0\leq i\leq d-1),& \qquad RU_d=0, \label{RU}\\
&LU_i\subseteq U_{i-1} \qquad\qquad (1\leq i\leq d),& \ \qquad LU_0=0. \label{LU}
\end{eqnarray}
\end{lemma}
\bigskip
%%%%%%%%%%%%%%%%%%%%%%%%%%%%%%%%%%%%%%%%%%%%%%%%%%

\begin{cor}\label{cor:Rtau}
The expression 
\begin{equation*}
R^{j-i}-\tau_{ij}(A)
\end{equation*}
vanishes on $U_i$ for $0\leq i\leq j \leq d+1$.
\end{cor}
\begin{proof}
Use (\ref{tau}), (\ref{RU}), and Lemma \ref{lemma:RA}.
\end{proof}

%%%%%%%%%%%%%%%%%%%%%%%%%%%%%%%%%%%%%%%%%%%%%%%%%%

\begin{lemma} \label{lemma:tauU_i}
For $0\leq i\leq j\leq d+1$, 
\begin{equation*}
\tau_{ij}(A)U_i\subseteq U_j.
\end{equation*}
\end{lemma}
\begin{proof}
Use Lemma \ref{lemma:RULU} and Corollary \ref{cor:Rtau}.
\end{proof}

\medskip

%%%%%%%%%%%%%%%%%%%%%%%%%%%%%%%%%%%%%%%%%%%%%%%%%%%%%%%%%%

The following result is a reformulation of \cite[Lemma 6.5]{Somealg}.  

\medskip

\begin{lemma}{\rm \cite[Lemma 6.5]{Somealg}} \label{lemma:6.5}
For $0\leq i \leq j\leq d$ the linear transformation
\begin{align*}
U_i&\to U_j  \\
v \ &\mapsto \tau_{ij}(A)v
\end{align*}
is an injection if $i+j\leq d$, a bijection if $i+j=d$, and a surjection if $i+j\geq d$.
\end{lemma}
\begin{proof}
By {\rm \cite[Lemma 6.5]{Somealg}} the linear transformation
$U_i\to U_j$, $v\mapsto R^{j-i}v$
is an injection if $i+j\leq d$, a bijection if $i+j=d$, and a surjection if $i+j\geq d$.
The result follows from this and Corollary \ref{cor:Rtau}.
\end{proof}

\begin{cor}\label{cor:6.5inj} 
The restriction of $A-\theta_{i}I$ to $U_i$ is injective for $0\leq i< d/2$.
\end{cor}

%%%%%%%%%%%%%%%%%%%%%%%%%%%%%%%%%%%%%%%%%%%%%%%%%%%%%%%%%%%%%%%%%%%%%%%%%%%%%%%%%%%%%%%%%%%%%%%%%%%%%%%%%%%%%%%%%%%%%%%%%%%%%%%%%%%%%%%%%%%%%%%%%%%%%%
\section{The second split decomposition of $V$}\label{section:Udd}

%%%%%%%%%%%%%%%%%%%%%%%%%%%%%%%%%%%%%%%%%%%%%%%%%%%%%%%%%%%%%%%%%%%%%%%%%%%%%%%%%%%%%%%%%%%%%%%%%%%%%%%%%%%%%%%%%%%%%%%%%%%%%%%%%%%%%%%%%%%%%%%%%%%%%
We continue to discuss the TD system $\Phi$ from Definition \ref{def:TDsys}.   
Since $\Phi^{\Downarrow}$ is a TD system on $V,$ all the results from Section \ref{section:U} apply to it.  For later use, we now emphasize a few of these results.
By definition,
\begin{equation}
U_i^{\Downarrow}=(E_0^*V+E_1^*V +\cdots + E_i^* V)\cap (E_0V+E_1V+\cdots + E_{d-i} V)\label{eq:Udd}
\end{equation}
for $0\leq i \leq d$. 
Applying Theorem \ref{thm:U} to $\Phi^{\Downarrow}$ we obtain the following facts.  
The subspaces $\{U_i^{\Downarrow}\}_{i=0}^d$ form a decomposition of $V$ which we call the {\it second split decomposition} of $V$.  
We also have that
\begin{align*}
&(A-\theta_{d-i} I)U_i^{\Downarrow}\subseteq U_{i+1}^{\Downarrow}&(0\leq i\leq d-1 ), \qquad &(A-\theta_0 I)U_d^{\Downarrow}=0,\\
&(A^*-\theta_i^* I)U_i^{\Downarrow}\subseteq U_{i-1}^{\Downarrow}&(1\leq i\leq d ), \qquad &(A^*-\theta_0^* I)U_0^{\Downarrow}=0.
\end{align*}
In addition, for $0\leq i \leq d$ both
\begin{align*}
U_i^{\Downarrow}+U_{i+1}^{\Downarrow}+\cdots + U_d^{\Downarrow}&=E_0V+E_{1}V+\cdots + E_{d-i}V,\\
U_0^{\Downarrow}+U_{1}^{\Downarrow}+\cdots + U_i^{\Downarrow} &=E_0^*V+E_{1}^*V+\cdots + E_i^*V.
\end{align*}

%%%%%%%%%%%%%%%%%%%%%%%%%%%%%%%%%%%%%%%%%%%%%%%%%%

\begin{lemma}\label{lemma:sumUUdd}
For $0\leq i \leq d$, 
\begin{eqnarray*}
U_0+U_1+\cdots +U_i=U_0^{\Downarrow}+U_1^{\Downarrow}+\cdots +U_i^{\Downarrow}.
\end{eqnarray*}
\end{lemma}
\begin{proof}
Both sides equal $E_0^*V+E_1^*V+\cdots + E_i^*V$ by Theorem \ref{thm:U}(iii).
\end{proof}

\medskip
%%%%%%%%%%%%%%%%%%%%%%%%%%%%%%%%%%%%%%%%%%%%%%%%%%

We now make some comments concerning $\{F_i^{\Downarrow}\}_{i=0}^d$ and $R^{\Downarrow}$.
For $0\leq i\leq d$, $F_i^{\Downarrow}$ is the projection of $V$ onto $U_i^{\Downarrow}$. 
Observe that
\begin{eqnarray}
R^{\Downarrow} = A-\sum_{h=0}^{d}\theta_{d-h} F^{\Downarrow}_h.\label{eq:defRdd}
\end{eqnarray}
For $0\leq i\leq j\leq d+1$, the action of $\left(R^{\Downarrow}\right)^{j-i}$ on $U_i^{\Downarrow}$ agrees with the action of $\eta_{ij}(A)$ on $U_i^{\Downarrow}$.
In addition, 
\begin{eqnarray*}
R^{\Downarrow}U_i^{\Downarrow}\subseteq U_{i+1}^{\Downarrow} \qquad (0\leq i\leq d-1), \qquad R^{\Downarrow}U_d^{\Downarrow}=0. \label{RUdd}
\end{eqnarray*}

%%%%%%%%%%%%%%%%%%%%%%%%%%%%%%%%%%%%%%%%%%%%%%%%%%
%%%%%%%%%%%%%%%%%%%%%%%%%%%%%%%%%%%%%%%%%%%%%%%%%%%%%%%%%%%%%%%%%%%%%%%%%%%%%%%%%%%%%%%%%%%%%%%%%%%%
\section{The projections $F_i$, $F_i^{\Downarrow}$}
%%%%%%%%%%%%%%%%%%%%%%%%%%%%%%%%%%%%%%%%%%%%%%%%%%
%%%%%%%%%%%%%%%%%%%%%%%%%%%%%%%%%%%%%%%%%%%%%%%%%%

We continue to discuss the TD system $\Phi$ from Definition \ref{def:TDsys}.
In this section, we consider how the maps $\{F_i\}_{i=0}^d$ and $\{F_i^{\Downarrow}\}_{i=0}^d$ interact.  
In \cite[Section 5]{Somealg}, there are a number of results concerning how the maps $\{E_i\}_{i=0}^d$ and $\{F_i\}_{i=0}^d$ interact.  The results given in this section are reformulations of these results.

\medskip
\begin{lemma}\label{lemma:projprod1}
For $0\leq i < j\leq d$ both
\begin{eqnarray}
F_jF_i^{\Downarrow}=0, \qquad\qquad F_j^{\Downarrow}F_i=0.\label{eq:projprod1}
\end{eqnarray}
\end{lemma}
\begin{proof}
We first verify the equation on the left in (\ref{eq:projprod1}).  By Definition \ref{def:F_i} and Lemma \ref{lemma:sumUUdd},
\begin{align}
F_jF_i^{\Downarrow}V &=F_jU_i^{\Downarrow}\nonumber\\
&\subseteq F_j(U_0^{\Downarrow}+U_1^{\Downarrow}+\cdots +U_i^{\Downarrow})\nonumber\\
&= F_j\left( U_0+U_1+\cdots+U_i\right).\label{eq:projprod2}
\end{align}
Since $i<j$, it follows from (\ref{eq:F_i2}) that (\ref{eq:projprod2}) equals 0.  So $F_jF_i^{\Downarrow}$ vanishes on $V$.

The proof for the equation on the right in (\ref{eq:projprod1}) is similar.
\end{proof}

\begin{lemma}\label{lemma:projprod2}
For $0\leq i\leq d$ both
\begin{align}
F_iF_i^{\Downarrow}F_i&=F_i,\label{eq:projprod2-1}\\
F_i^{\Downarrow}F_iF_i^{\Downarrow}&=F_i^{\Downarrow}.\label{eq:projprod2-2}
\end{align}
\end{lemma}
\begin{proof}
We first show (\ref{eq:projprod2-1}).  By Lemma \ref{lemma:F} and Lemma \ref{lemma:projprod1},
\begin{equation*}
F_i=F_iF_i=F_i\left(\sum_{h=0}^d F_h^{\Downarrow}\right)F_i
=F_iF_i^{\Downarrow}F_i.
\end{equation*}

The proof of (\ref{eq:projprod2-2}) is similar.  
\end{proof}

\begin{lemma}\label{lemma:F-inv}
For $0\leq i\leq d$ the restrictions
\begin{equation*}
F_i^{\Downarrow}|_{U_i}:U_i\to U_i^{\Downarrow}, \qquad \qquad 
F_i|_{U_i^{\Downarrow}}:U_i^{\Downarrow}\to U_i\\
\end{equation*}
are bijections.  Moreover, these bijections are inverses.
\end{lemma}
\begin{proof}
We first show that the map $F_iF_i^{\Downarrow}$ acts as the identity on $U_i$.  Let $v\in U_i$.  By (\ref{eq:F_i1}) and (\ref{eq:projprod2-1}),
\begin{equation*}
F_iF_i^{\Downarrow}v=F_iF_i^{\Downarrow}F_iv
=F_iv
=v.
\end{equation*}
We have shown $F_iF_i^{\Downarrow}$ acts as the identity on $U_i$.  
One can show similarly that $F_i^{\Downarrow}F_i$ acts as the identity on $U_i^{\Downarrow}$.  The result follows.
\end{proof}

\begin{lemma}{\rm \cite[Lemma 6.4]{Somealg} }\label{lemma:FR}
We have
\begin{enumerate}
\item[{\rm (i)}]
$RF_i=F_{i+1}R\qquad\qquad (0\leq i\leq d-1),\qquad RF_d=0,\qquad F_0R=0.$
\item[{\rm (ii)}]
$LF_i=F_{i-1}L\qquad\qquad (1\leq i\leq d),\qquad LF_0=0,\qquad F_dL=0.
$\
\end{enumerate}
\end{lemma}
\bigskip

\begin{lemma}\label{lemma:RRddF}
For $0\leq i\leq d-1$,
\begin{equation*}
R^{\Downarrow}F_i^{\Downarrow}F_i=F_{i+1}^{\Downarrow}F_{i+1}R.
\end{equation*}
\end{lemma}
\begin{proof}
We show $R^{\Downarrow}F_i^{\Downarrow}F_i-F_{i+1}^{\Downarrow}F_{i+1}R=0$.  
By Lemma \ref{lemma:FR}(i) (applied to both $\Phi$ and $\Phi^{\Downarrow}$), 
\begin{align}
R^{\Downarrow}F_i^{\Downarrow}F_i-F_{i+1}^{\Downarrow}F_{i+1}R&=F_{i+1}^{\Downarrow}R^{\Downarrow}F_i-F_{i+1}^{\Downarrow}RF_{i}\nonumber\\
&=F_{i+1}^{\Downarrow}\left(R^{\Downarrow}-R\right)F_i.\label{eq:RRddF1}
\end{align}
By Definition \ref{def:R},
\begin{equation}
R^{\Downarrow}-R=\sum_{h=0}^d\theta_hF_h-\sum_{h=0}^d\theta_{d-h}F_h^\Downarrow.\label{eq:RRddF2}
\end{equation}
Eliminate $R^{\Downarrow}-R$ in (\ref{eq:RRddF1}) using (\ref{eq:RRddF2}).  Simplify the resulting expression using (\ref{eq:FF}) (applied to both $\Phi$ and $\Phi^{\Downarrow}$) and Lemma \ref{lemma:projprod1} to get 0.  
\end{proof}

%%%%%%%%%%%%%%%%%%%%%%%%%%%%%%%%%%%%%%%%%%%%%%%%%%
%%%%%%%%%%%%%%%%%%%%%%%%%%%%%%%%%%%%%%%%%%%%%%%%%%
%%%%%%%%%%%%%%%%%%%%%%%%%%%%%%%%%%%%%%%%%%%%%%%%%%
\section{The subspaces $K_i$ }
%%%%%%%%%%%%%%%%%%%%%%%%%%%%%%%%%%%%%%%%%%%%%%%%%%
%%%%%%%%%%%%%%%%%%%%%%%%%%%%%%%%%%%%%%%%%%%%%%%%%%
%%%%%%%%%%%%%%%%%%%%%%%%%%%%%%%%%%%%%%%%%%%%%%%%%%

We continue to discuss the TD system $\Phi$ from Definition \ref{def:TDsys}.
Shortly we will define the linear transformation $\Psi$.  In our discussion of $\Psi$, it will be useful to consider a certain refinement of the first and second split decomposition of $V$.  This refinement was introduced in \cite{refinement}.  
In order to describe this refinement, we introduce a sequence of subspaces $\{K_i\}_{i=0}^r$, where $r=\lfloor d/2\rfloor$.\\

\begin{definition}\label{def:K}{\rm
For $0\leq i\leq d/2$, define the subspace $K_i\subseteq V$ by
\begin{equation*}
K_i=(E_0^*V+E_1^*V+\cdots+E_i^*V)\cap (E_iV+E_{i+1}V+\cdots+E_{d-i}V).
\end{equation*}
Observe that $K_0=E_0^*V=U_0$.}
\end{definition}

%%%%%%%%%%%%%%%%%%%%%%%%%%%%%%%%%%%%%%%%%%%%%%%%%%%%%%%%%%%%%

\begin{lemma}\label{lemma:K} 
We have 
\begin{equation*}
K_i=U_i\cap U_i^{\Downarrow}\qquad\qquad \qquad(0\leq i\leq d/2).
\end{equation*}
\end{lemma}
\begin{proof}
Use (\ref{eq:Udd}), Definition \ref{def:U}, and Definition \ref{def:K}.
\end{proof}

\medskip

\begin{lemma} {\rm \cite[Lemma 4.1(iii)]{refinement}} \label{lemma:K2}
For $0\leq i\leq d/2$, 
the restriction of $\tau_{i,d-i+1}(A)$ to $U_i$ has kernel $K_i$.  
\end{lemma}
\begin{proof}  Use Lemma \ref{lemma:taukernel} and Definition \ref{def:U}. 
\end{proof}

\medskip
\medskip
%%%%%%%%%%%%%%%%%%%%%%%%%%%%%%%%%%%%%%%%

We now consider the spaces
\begin{equation*}
\tau_{ij}(A)K_i
\end{equation*}
where $0\leq i\leq d/2$ and  $i\leq j\leq d-i$. 
We start with an observation.

\medskip

%%%%%%%%%%%%%%%%%%%%%%%%%%%%%%%%%%%%%%%%

\begin{lemma}{\rm \cite[Lemma 4.1(vi)]{refinement}} \label{lemma:tauK}
For $0\leq i\leq d/2$ and $i\leq j\leq d-i$, the linear transformation
\begin{align*}
K_i&\to \tau_{ij} (A)K_i\\
v&\mapsto  \tau_{ij}(A)v
\end{align*}
is a bijection.
\end{lemma}
\begin{proof}
By construction the map is surjective.  
By Lemma \ref{lemma:6.5} the restriction of $\tau_{ij}(A)$ to $K_i$ is injective. 
The result follows.
\end{proof}

\medskip
%%%%%%%%%%%%%%%%%%%%%%%%%%%%%%%%%%%%%%%%

From Lemma \ref{lemma:tauK}, we draw two corollaries.
 
 \medskip
 
\begin{cor}
For $0\leq i\leq d/2$ and $i\leq j\leq k \leq d-i$, the linear transformation
\begin{align*}
\tau_{ij}(A)K_i&\to \tau_{ik} (A)K_i\\
v&\mapsto  \tau_{jk}(A)v
\end{align*}
is a bijection.
\end{cor}
\begin{proof}
Use Lemma \ref{lemma:tauK} and the equation on the left in (\ref{eq:taumult}). 
\end{proof}

%%%%%%%%%%%%%%%%%%%%%%%%%%%%%%%%%%%%%%%%

\begin{cor}
For $0\leq i\leq d/2$ and $i\leq j\leq d-i$, the dimension of $\tau_{ij}(A)K_i$ coincides with the dimension of $K_i$.
\end{cor}

%%%%%%%%%%%%%%%%%%%%%%%%%%%%%%%%%%%%%%%%
%%%%%%%%%%%%%%%%%%%%%%%%%%%%%%%%%%%%%%%%
%%%%%%%%%%%%%%%%%%%%%%%%%%%%%%%%%%%%%%%%

\section{Concerning the decomposition $\{U_i\}_{i=0}^d$}

%%%%%%%%%%%%%%%%%%%%%%%%%%%%%%%%%%%%%%%%
%%%%%%%%%%%%%%%%%%%%%%%%%%%%%%%%%%%%%%%%
%%%%%%%%%%%%%%%%%%%%%%%%%%%%%%%%%%%%%%%%
We continue to discuss the TD system $\Phi$ from Definition \ref{def:TDsys}.  
Recall the first split decomposition $\{U_i\}_{i=0}^d$ of $V$ from Definition \ref{def:U}.  
We know that $K_0=U_0$ and $K_i\subseteq U_i$ for $1\leq i \leq d$.  
We will use this fact along with information about the raising map $R$ to give a decomposition of each $U_i$.\\

The following result is essentially due to K. Nomura \cite[Theorem 4.2]{refinement}.  We give an alternate proof.

\medskip

\begin{lemma} {\rm \cite[Theorem 4.2]{refinement}} \label{lemma:U1}
For $1\leq i\leq d/2$, each of the following sums is direct.
\begin{enumerate}
\item[{\rm (i)}] $\displaystyle U_{i}=K_{i}+RU_{i-1}$,
\item[{\rm (ii)}] $\displaystyle U_{i}=K_{i}+(A-\theta_{i-1} I)U_{i-1}$.
\end{enumerate}
\end{lemma}
\begin{proof}
(i)  We first show that $\displaystyle U_{i}=K_{i}+RU_{i-1}$.  
By Lemma \ref{lemma:RULU} and Lemma \ref{lemma:K}, $\displaystyle U_i\supseteq K_{i}+RU_{i-1}$. 
We now show $\displaystyle U_i\subseteq K_{i}+RU_{i-1}$.  Let $v\in U_{i}$.  By Lemma \ref{lemma:RULU} we get $R^{d-2i+1} v\in U_{d-i+1}$.  
By Corollary \ref{cor:Rtau} and Lemma \ref{lemma:6.5} there exists $w\in U_{i-1}$ such that $R^{d-2i+2}w=R^{d-2i+1} v$.  Rearranging terms we obtain $R^{d-2i+1}(Rw- v)=0$.  So $Rw-v$ is in the kernel of $R^{d-2i+1}$.  By Lemma \ref{lemma:RULU}, $Rw-v\in U_i$.  
By Corollary \ref{cor:Rtau} and Lemma \ref{lemma:K2}, $K_i$ is the intersection of $U_i$ and the kernel of $R^{d-2i+1}$.  
By these comments $Rw-v\in K_i$.  Therefore
\begin{align*}
v&= -(Rw-v)+Rw\\
&\in  K_i+RU_{i-1}.
\end{align*}
Hence $\displaystyle U_{i}\subseteq K_{i}+RU_{i-1}$.  
We have shown $U_i=K_i+RU_{i-1}$.
We now show that this sum is direct.  Let $v\in K_i\cap RU_{i-1}$.  Since $v\in RU_{i-1}$, there exists $w\in U_{i-1}$ such that $v=Rw$.  
Recall $v\in K_i$ so $R^{d-2i+1}v=0$.  
Therefore $R^{d-2i+2}w=0$.  By Lemma \ref{lemma:6.5}, the restriction of $R^{d-2i+2}$ to $U_{i-1}$ is injective.  So $w=0$ and thus $v=0$.
We have shown that the sum $U_i=K_i+RU_{i-1}$ is direct.

(ii)  Use (i) and Lemma \ref{lemma:RA}.
\end{proof}

\medskip

%%%%%%%%%%%%%%%%%%%%%%%%%%%%%%%%%%%%%%%%

From Lemma \ref{lemma:U1} we obtain the following two corollaries.

\medskip

\begin{cor}{\rm \cite[Corollary 6.6]{Somealg}}\label{cor:dim} 
With reference to Lemma \ref{lemma:dim1}, 
\begin{enumerate}
\item[{\rm (i)}] $\rho_i \leq \rho_{i+1}$ \ for \  $0\leq i <d/2$,
\item[{\rm (ii)}] $\rho_i \geq \rho_{i+1}$  \ for \  $d/2 \leq i \leq d-1$.
\end{enumerate}
\end{cor}
\begin{proof}
(i)  Use Lemma \ref{lemma:U1}(i) and Lemma \ref{lemma:6.5}.\\
(ii)  Use Corollary \ref{cor:dim}(i) and Lemma \ref{lemma:dim1}.
\end{proof}

\begin{cor}{\rm \cite[Lemma 4.3]{refinement}}
For $1\leq i\leq d/2$, the dimension of $K_i$ equals $\rho_i-\rho_{i-1}$ (this dimension could be zero).  Moreover, the dimension of $K_0$ equals $\rho_0$.
\end{cor}
%%%%%%%%%%%%%%%%%%%%%%%%%%%%%%%%%%%%%%%%

\begin{lemma}{\rm \cite[Theorem 4.7]{refinement}}  \label{lemma:U} 
\begin{enumerate}
\item[{\rm (i)}] For $0\leq i\leq d/2$, the following sum is direct.
\begin{equation}
U_i=K_i+\tau_{i-1,i}(A)K_{i-1}+\tau_{i-2,i}(A)K_{i-2}+\cdots+\tau_{0,i}(A)K_{0}.\label{eq:UK1}
\end{equation}

\item[{\rm (ii)}] For $d/2\leq i\leq d$, the following sum is direct.
\begin{equation*}
U_i=\tau_{d-i,i}(A)K_{d-i}+\tau_{d-i-1,i}(A)K_{d-i-1}+\cdots+\tau_{0,i}(A)K_{0}.\\
\end{equation*}
\end{enumerate}
\end{lemma}
\begin{proof}
(i)  Recall $U_0=K_0$.  By Lemma \ref{lemma:U1}(ii), the sum $U_j=K_j+(A-\theta_{j-1} I)U_{j-1}$ is direct for $1\leq j\leq i$.
Combining these equations and simplifying the result using (\ref{tau}), we get (\ref{eq:UK1}).
The directness of the sum (\ref{eq:UK1}) follows in view of Corollary \ref{cor:6.5inj}.\\
(ii)  Observe that $0\leq d-i\leq d/2$.  So (\ref{eq:UK1}) gives a decomposition of $U_{d-i}$.  
By Lemma \ref{lemma:6.5}, the restriction of $\tau_{d-i,i}(A)$ to $U_{d-i}$ gives a bijection $U_{d-i}\to U_i$.   
Apply this bijection to each term in the above mentioned decomposition for $U_{d-i}$ and simplify the result using the equation on the left in (\ref{eq:taumult}).
\end{proof}

\medskip

Combining parts (i) and (ii) of Lemma \ref{lemma:U} we have 
\begin{equation}
U_j=\sum_{i=0}^{\min\{j,d-j\}} \tau_{ij}(A) K_i \qquad (\text{direct sum})\label{eq:Udecomp}
\end{equation}
for $0\leq j \leq d$.
\medskip
\begin{cor} {\rm \cite[Theorem 4.8]{refinement}} \label{cor:Urefine}
The following sum is direct.
\begin{eqnarray}
V=\sum_{i=0}^{r}\sum_{j=i}^{d-i} \tau_{ij}(A)K_i,\label{eq:Urefine1}
\end{eqnarray}
where $r=\lfloor d/2\rfloor$.
\end{cor}
\begin{proof}
In the decomposition of $V$ from Theorem \ref{thm:U}, evaluate each summand using (\ref{eq:Udecomp}).  In the resulting double summation, invert the order of summation.
\end{proof}

%%%%%%%%%%%%%%%%%%%%%%%%%%%%%%%%%%%%%%%%
%%%%%%%%%%%%%%%%%%%%%%%%%%%%%%%%%%%%%%%%
%%%%%%%%%%%%%%%%%%%%%%%%%%%%%%%%%%%%%%%%

\section{The subalgebra $M$}

%%%%%%%%%%%%%%%%%%%%%%%%%%%%%%%%%%%%%%%%
%%%%%%%%%%%%%%%%%%%%%%%%%%%%%%%%%%%%%%%%
%%%%%%%%%%%%%%%%%%%%%%%%%%%%%%%%%%%%%%%%
We continue to discuss the TD system $\Phi$ from Definition \ref{def:TDsys}.
Recall from Section \ref{section:prelim} the subalgebra $M$ of ${\rm End}(V)$ generated by $A$.  In our discussion of $M$, we mentioned that each of $\{E_i\}_{i=0}^d$, $\{A^i\}_{i=0}^d$ is a basis for $M$.  In this section, we give a third basis for $M$ and use it to realize $V$ as a direct sum of $M$-modules.\\

%%%%%%%%%%%%%%%%%%%%%%%%%%%%%%%%%%%%%%%%

\begin{lemma}\label{lemma:Mbasis}For $0\leq i\leq d/2$, the vector space $M$ has basis
\begin{equation}
\{ E_0, E_1, \mathellipsis,E_{i-1}\}\cup\{ E_{d-i+1}, E_{d-i+2}, \mathellipsis,E_{d}\} \cup\{ \tau_{ij}(A)|i\leq j\leq d-i\}. \label{eqn:Mbasis}
\end{equation}
\end{lemma}

\begin{proof} By \cite[Lemma 5.1]{TDclass}, 
\begin{equation*}
\{ E_0, E_1, \mathellipsis,E_{i-1}\}\cup\{ E_{d-i+1}, E_{d-i+2}, \mathellipsis,E_{d}\} \cup\{ A^{j-i}|i\leq j\leq d-i\}
\end{equation*}
is a basis for $M$.  By the comments following (\ref{eta}),
\begin{equation*}
\text{Span}\{A^{j-i}|i\leq j\leq d-i\} = \text{Span}\{\tau_{ij}(A)|i\leq j\leq d-i\}.
\end{equation*}
The result follows.
\end{proof}

\medskip

%%%%%%%%%%%%%%%%%%%%%%%%%%%%%%%%%%%%%%%%

For the rest of this section, we view $V$ as an $M$-module.  
For $0\leq i\leq d/2$ let $MK_i$ denote the $M$-submodule of $V$ generated by $K_i$.
Our goal in this section is to show that the sum $V=\sum_{i=0}^{r} MK_i$ is direct, where $r=\lfloor d/2\rfloor$.  
We start by giving a detailed description of the $MK_i$.

\medskip  

%%%%%%%%%%%%%%%%%%%%%%%%%%%%%%%%%%%%%%%%

\begin{lemma}\label{lemma:MK}
For $0\leq i\leq d/2$ such that $K_i\neq 0$, the sum
\begin{equation}
MK_i=K_i+\tau_{i,i+1}(A)K_i+\tau_{i,i+2}(A)K_i+\cdots+\tau_{i,d-i}(A)K_i\label{eq:MK1}
\end{equation}
is direct.  Moreover $\tau_{i,d-i+1}$ is the minimal polynomial for the action of $A$ on $MK_i$.
\end{lemma}

\begin{proof}
For the basis of $M$ given in (\ref{eqn:Mbasis}), apply each element to $K_i$.
By Definition \ref{def:K}, each primitive idempotent in (\ref{eqn:Mbasis}) vanishes on $K_i$.
This gives equation (\ref{eq:MK1}).
We now show the sum on the right in (\ref{eq:MK1}) is direct.  By Lemma \ref{lemma:tauU_i},
 we have $\tau_{ij}(A)K_i\subseteq U_{j}$ for $i\leq j \leq d-i$. 
The sum (\ref{eq:MK1}) is direct by this and Theorem \ref{thm:U}.

It remains to show that $\tau_{i,d-i+1}$ is the minimal polynomial for the action of $A$ on $MK_i$.  Let $P$ denote the minimal polynomial for the action of $A$ on $MK_i$ and let $k$ denote the degree of $P$.  
By Lemma \ref{lemma:taukernel} and Definition \ref{def:K}, $\tau_{i,d-i+1}(A)K_i=0$.
Since $A\in M$ and $M$ is commutative, it follows that $ \tau_{i,d-i+1}(A) MK_i=0$.
So $P$ divides $\tau_{i,d-i+1}$ and hence $k\leq d-2i+1$.

Suppose now that $k<d-2i+1$ to get a contradiction.
Since the degree of $P$ is $k$,
\begin{equation}
MK_i=K_i+AK_i+\cdots +A^{k-1}K_i.\label{eq:MK22}
\end{equation}
By (\ref{eq:A^kU}), the right-hand side of (\ref{eq:MK22}) is contained in $U_i+U_{i+1}+\cdots +U_{i+k-1}$. 
By Lemma \ref{lemma:tauK}, the restriction of $\tau_{i,d-i}(A)$ to $K_i$ is an injection. It follows from this and $K_i\neq 0$ that $\tau_{i,d-i}(A)K_i\neq 0$.  Recall that $\tau_{i,d-i}(A)K_i\subseteq U_{d-i}$.  
By (\ref{eq:MK1}) and the above comments we find that
$\tau_{i,d-i}(A)K_i$ is contained in the intersection of 
$U_i+U_{i+1}+\cdots+U_{i+k-1}$ and $U_{d-i}$.  Since $k<d-2i+1$, this intersection is zero by Theorem \ref{thm:U}.  
Therefore $\tau_{i,d-i}(A)K_i=0$ for a contradiction.  
Thus $k=d-2i+1$ and therefore
$P=\tau_{i,d-i+1}$ since $\tau_{i,d-i+1}$ is monic.
\end{proof}

%%%%%%%%%%%%%%%%%%%%%%%%%%%%%%%%%%%%%%%%%%%%%%%%%%%%%%%%%%

\begin{cor}\label{cor:Mv1}
For $0\leq i\leq d/2$ and $0\neq v\in K_i$, the vector space $Mv$ has basis
\begin{equation*}
v,\quad \tau_{i,i+1}(A)v,\quad \tau_{i,i+2}(A)v,\quad \mathellipsis,\quad \tau_{i,d-i}(A)v.
\end{equation*}
\end{cor}

%%%%%%%%%%%%%%%%%%%%%%%%%%%%%%%%%%%%%%%%%%%%%%%%%%%%%%%%%%

\begin{lemma}\label{lemma:VMK_i}  
The following is a direct sum of $M$-modules.
\begin{eqnarray}
&V = \displaystyle\sum_{i=0}^{r} MK_i,\label{eq:VMK_i}
\end{eqnarray}
where $r=\lfloor d/2\rfloor$.
\end{lemma}
\begin{proof}
Equation (\ref{eq:VMK_i}) follows from Corollary \ref{cor:Urefine} and Lemma \ref{lemma:MK}.
The directness of the sum follows from the directness of the sum in Corollary \ref{cor:Urefine}.
\end{proof}

%%%%%%%%%%%%%%%%%%%%%%%%%%%%%%%%%%%%%%%%%%%%%%%%%%%%%%%%%%%%%%%%%%%%%%%%%%%%%%%%%%%%%%%%%%%%%%%%%%%%%%%%%%%%%%%%%%%%%%%%

\section{The linear transformation $\Delta$}\label{section:Delta}

%%%%%%%%%%%%%%%%%%%%%%%%%%%%%%%%%%%%%%%%%%%%%%%%%%%%%%%%%%%%%%%%%%%%%%%%%%%%%%%%%%%%%%%%%%%%%%%%%%%%%%%%%%%%%%%%%%%%%%%%
We continue to discuss the TD system $\Phi$ from Definition \ref{def:TDsys}.  In this section we will construct a linear transformation $\Delta\in {\rm End}(V)$ that has certain properties which we find attractive.  It will turn out that $\Delta$ is the unique element of ${\rm End} (V)$ such that both
\begin{align*}
&(\Delta -I)E_i^{*} V\subseteq E_0^{*}V+E_1^{*}V+\cdots +E_{i-1}^{*}V,\\
&\Delta (E_iV+E_{i+1}V+\cdots+E_d V)= E_0 V+E_{1}V+\cdots+E_{d-i} V
\end{align*}
for $0\leq i\leq d$.

\smallskip
%%%%%%%%%%%%%%%%%%%%%%%%%%%%%%%%%%%%%%%%%

\begin{definition}\label{def:Delta}{\rm
Define $\Delta\in{\rm End}(V)$ by
\begin{equation*}
\Delta = \sum_{h=0}^d F_h^{\Downarrow}F_h,
\end{equation*}
where $F_h$, $F_h^{\Downarrow}$ are from Definition \ref{def:F_i}.
}
\end{definition}

%%%%%%%%%%%%%%%%%%%%%%%%%%%%%%%%%%%%%%%%%
\begin{lemma}\label{lemma:FDelta}
With reference to Definition \ref{def:Delta}, 
\begin{equation*}
F_i^{\Downarrow}\Delta=\Delta F_i\qquad\qquad (0\leq i\leq d).
\end{equation*}
\end{lemma}
\begin{proof}
Use (\ref{eq:FF}) and Definition \ref{def:Delta}.
\end{proof}
%%%%%%%%%%%%%%%%%%%%%%%%%%%%%%%%%%%%%%%%%

\begin{lemma}\label{lemma:Delta^(-1)}
With reference to Definition \ref{def:Delta}, $\Delta^{-1}$ exists and 
\begin{equation*}
\Delta^{-1}=\Delta^{\Downarrow}.
\end{equation*}
\end{lemma}
\begin{proof}
Observe that $\Delta^{\Downarrow}=\sum_{h=0}^d F_hF_h^{\Downarrow}$.  Consider the product $\Delta\Delta^{\Downarrow}$.  Simplify this product using Lemma \ref{lemma:F} and Lemma \ref{lemma:projprod2} to obtain $\Delta\Delta^{\Downarrow}=I$.
\end{proof}

%%%%%%%%%%%%%%%%%%%%%%%%%%%%%%%%%%%%%%%%%%%%%%%%%%%%%%%%%%
\begin{lemma}\label{lemma:DeltaU}
With reference to Definition \ref{def:Delta}, 
\begin{align}
&\Delta U_i=U_i^{\Downarrow} &\ (0\leq i \leq d),\label{eq:DeltaU1}\\
&(\Delta -I)U_i\subseteq U_0+U_1+\cdots +U_{i-1} &\ (0\leq i \leq d).\label{eq:DeltaU2}
\end{align}
\end{lemma}
\begin{proof}
Line (\ref{eq:DeltaU1}) follows from Definition \ref{def:F_i}, Lemma \ref{lemma:F-inv} and Definition \ref{def:Delta}.

We now verify (\ref{eq:DeltaU2}).  By Definition \ref{def:F_i}, it suffices to show that $F_j\left(\Delta-I\right)U_i=0$ for $i\leq j\leq d$.  For $i=j$, this follows from Definition \ref{def:F_i}, Definition \ref{def:Delta}, and (\ref{eq:projprod2-1}).  For $i+1\leq j \leq d$, this follows from Definition \ref{def:F_i}, Definition \ref{def:Delta},  and (\ref{eq:projprod1}).
\end{proof}

\medskip

We now show that (\ref{eq:DeltaU1}), (\ref{eq:DeltaU2}) characterize $\Delta$.

\medskip

\begin{lemma}\label{lemma:DD'}
Given $\Delta'\in{\rm End}(V)$ such that
\begin{align}
&\Delta' U_i\subseteq U_i^{\Downarrow} &\ (0\leq i \leq d),\label{eq:DeltaU21}\\
&(\Delta' -I)U_i\subseteq U_0+U_1+\cdots +U_{i-1} &\ (0\leq i \leq d).\label{eq:DeltaU22}
\end{align}
Then $\Delta'=\Delta$.
\end{lemma}
\begin{proof}
In view of Theorem \ref{thm:U}, it suffices to show that $\Delta$, $\Delta'$ agree on $U_i$ for $0\leq i\leq d$.  Let $i$ be given.  By (\ref{eq:DeltaU1}) and (\ref{eq:DeltaU21}), 
\begin{equation}
\left(\Delta-\Delta'\right)U_i\subseteq U_i^{\Downarrow}.\label{eq:DeltaU3}
\end{equation}
By (\ref{eq:DeltaU2}), (\ref{eq:DeltaU22}), and Lemma \ref{lemma:sumUUdd},
\begin{align}
\left(\Delta-\Delta'\right)U_i&\subseteq U_0+U_1+\cdots+U_{i-1}\nonumber\\
&= U_0^{\Downarrow}+U_1^{\Downarrow}+\cdots+U_{i-1}^{\Downarrow}.
\label{eq:DeltaU4}
\end{align}
Combining (\ref{eq:DeltaU3}) and (\ref{eq:DeltaU4}) we find that $\left(\Delta-\Delta'\right)U_i$ is contained in the intersection of $U_i^{\Downarrow}$ and $U_0^{\Downarrow}+U_1^{\Downarrow}+\cdots +U_{i-1}^{\Downarrow}$.  This intersection is zero by Theorem \ref{thm:U} (applied to $\Phi^{\Downarrow}$).  Therefore $(\Delta-\Delta')U_i=0$.  So $\Delta$, $\Delta'$ agree on $U_i$.  
\end{proof}
%%%%%%%%%%%%%%%%%%%%%%%%%%%%%%%%%%%%%%%%%%%%%%%%%%%%%%%%%%

\begin{lemma}\label{lemma:Delta^{-1}U}
With reference to Definition \ref{def:Delta}, 
\begin{align*}
&(\Delta^{-1} -I)U_i\subseteq U_0+U_1+\cdots +U_{i-1} &\ (0\leq i \leq d).
\end{align*}
\end{lemma}
\begin{proof}
Apply $\Delta^{-1}$ to both sides in (\ref{eq:DeltaU2}).  In the resulting containment, simplify the right-hand side using Lemma \ref{lemma:sumUUdd} and (\ref{eq:DeltaU1}).
\end{proof}

\begin{lemma}
With reference to Definition \ref{def:Delta}, 
\begin{align*}
&(\Delta -I)U_i^{\Downarrow}\subseteq U_0^{\Downarrow}+U_1^{\Downarrow}+\cdots +U_{i-1}^{\Downarrow} &\ (0\leq i \leq d).
\end{align*}
\end{lemma}
\begin{proof}
Apply Lemma \ref{lemma:Delta^{-1}U} to $\Phi^{\Downarrow}$.  Use Lemma \ref{lemma:Delta^(-1)} to simplify the result.
\end{proof}

\medskip

We now obtain the characterization of $\Delta$ given in the Introduction.

\begin{lemma}\label{lemma:DeltaEV}
With reference to Definition \ref{def:Delta}, 
\begin{align}
&(\Delta -I)E_i^{*} V\subseteq E_0^{*}V+E_1^{*}V+\cdots +E_{i-1}^{*}V &\ (0\leq i \leq d),\label{Delta1}\\
&\Delta (E_iV+E_{i+1}V+\cdots+E_d V)= E_0 V+E_{1}V+\cdots+E_{d-i} V &(0\leq i \leq d). \label{Delta2}
\end{align}
\end{lemma} 
\begin{proof}
We first show (\ref{Delta1}).  
By Theorem \ref{thm:U}(iii) and (\ref{eq:DeltaU2}), 
\begin{align*}
(\Delta - I)E_i^* V &\subseteq (\Delta - I)(E_0^* V+E_1^* V+\cdots +E_i^* V )\\
&= (\Delta - I)(U_0+U_1+\cdots +U_i)\\
&\subseteq U_0+U_1+\cdots +U_{i-1}\\
&=E_0^* V+E_1^* V+\cdots +E_{i-1}^* V.
\end{align*}
We now show (\ref{Delta2}).  Applying Theorem \ref{thm:U}(iii) to both $\Phi$ and $\Phi^{\Downarrow}$, and also using (\ref{eq:DeltaU1}), we obtain
\begin{align*}
\Delta (E_iV+E_{i+1}V+\cdots+E_d V)&= \Delta (U_i+U_{i+1}+\cdots+U_d)\\
&= U_i^{\Downarrow}+U_{i+1}^{\Downarrow}+\cdots+U_d^{\Downarrow}\\
&= E_0V+E_{1}V+\cdots+E_{d-i} V.
\end{align*}
\end{proof}

\medskip

We now show that (\ref{Delta1}), (\ref{Delta2}) characterize $\Delta$.

\medskip

\begin{lemma}
Given $\Delta'\in{\rm End}(V)$ such that 
\begin{align}
&(\Delta' -I)E_i^{*} V\subseteq E_0^{*}V+E_1^{*}V+\cdots +E_{i-1}^{*}V &\ (0\leq i \leq d),\label{Delta21}\\
&\Delta' (E_iV+E_{i+1}V+\cdots+E_d V)\subseteq E_0 V+E_{1}V+\cdots+E_{d-i} V &(0\leq i \leq d). \label{Delta22}
\end{align}
Then $\Delta'=\Delta$.
\end{lemma}
\begin{proof}
By Lemma \ref{lemma:DD'}, it suffices to show that $\Delta'$ satisfies (\ref{eq:DeltaU21}) and (\ref{eq:DeltaU22}).  These lines are routinely verified using Theorem \ref{thm:U}(iii) (applied to both $\Phi$ and $\Phi^{\Downarrow}$) and Lemma \ref{lemma:sumUUdd}.
\end{proof}
%%%%%%%%%%%%%%%%%%%%%%%%%%%%%%%%%%%%%%%%%%%%%%%%%%%%%%%%%%

\medskip

We now derive some relations involving $\Delta$ that will be of use later.

\begin{lemma}\label{lemma:RDelta}
With reference to Definition \ref{def:Delta}, 
\begin{equation*}
R^{\Downarrow }\Delta=\Delta R.
\end{equation*}
\end{lemma}
\begin{proof}
In the expression $R^{\Downarrow}\Delta-\Delta R$, eliminate $\Delta$ using Definition \ref{def:Delta}.  Simplify the result using Lemma \ref{lemma:FR}(i) and Lemma \ref{lemma:RRddF} to obtain $R^{\Downarrow}\Delta-\Delta R=0$.
\end{proof}

\medskip

\begin{lemma}\label{lemma:DeltaA-ADelta}
With reference to Definition \ref{def:Delta}, \begin{equation}
\Delta A - A\Delta = 
\sum_{h=0}^{d} ( \theta_h - \theta_{d-h}) F_h^{\Downarrow}F_h.\label{eq:DeltaA-ADelta}
\end{equation}
\end{lemma}
\begin{proof}
By Lemma \ref{lemma:RDelta},
\begin{equation}
\Delta R-R^{\Downarrow}\Delta=0.\label{eq:ADeltaPsiR2}
\end{equation}
In (\ref{eq:ADeltaPsiR2}), eliminate $R$ and $R^{\Downarrow}$ using  Definition \ref{def:R} and (\ref{eq:defRdd}) to get
\begin{equation}
\Delta A-A\Delta=\sum_{h=0}^d \theta_h\Delta F_h-\sum_{h=0}^d\theta_{d-h}F_h^{\Downarrow}\Delta.\label{eq:ADeltaPsiR3}
\end{equation}
Simplify the right-hand side of (\ref{eq:ADeltaPsiR3}) using Definition \ref{def:Delta} and (\ref{eq:FF}) to get the result.
\end{proof}

\medskip

We now express Lemma \ref{lemma:DeltaA-ADelta} from a slightly different perspective.

\begin{cor}\label{cor:ADelta-2}
With reference to Definition \ref{def:Delta}, \begin{equation*}
A - \Delta^{-1}A\Delta = 
\sum_{h=0}^{d} ( \theta_h - \theta_{d-h}) F_h.
\end{equation*}
\end{cor}
\begin{proof}
Apply $\Delta^{-1}$ to both sides of (\ref{eq:DeltaA-ADelta}).  
Simplify the resulting right-hand side using Lemma \ref{lemma:projprod2}, Definition \ref{def:Delta}, and (\ref{eq:FF}).
\end{proof}

\medskip

\begin{lemma}\label{lemma:LDelta}
With reference to Definition \ref{def:Delta}, 
\begin{equation}
L^{\Downarrow}\Delta-\Delta L=A^{*}\Delta-\Delta A^{*}.\label{eq:LDelta1}
\end{equation}
\end{lemma}
\begin{proof}
In the left-hand side of (\ref{eq:LDelta1}), eliminate $L$ and $L^{\Downarrow}$ using Definition \ref{def:R}.  Evaluate the result using Lemma \ref{lemma:FDelta}.
\end{proof}

\medskip

\begin{lemma}\label{lemma:DeltaA^*U}
With reference to Definition \ref{def:Delta}, 
\begin{equation*}
(\Delta^{-1}A^{*}\Delta- A^{*})U_i\subseteq U_{i-1}\qquad\qquad (1\leq i\leq d),\qquad 
(\Delta^{-1}A^{*}\Delta- A^{*})U_0=0.
\end{equation*}
\end{lemma}
\begin{proof}
By Lemma \ref{lemma:LDelta}, 
\begin{align*}
\Delta^{-1}A^{*}\Delta- A^{*}=\Delta^{-1}L^{\Downarrow}\Delta- L.
\end{align*}
Let $1\leq i\leq d$.  
By (\ref{eq:DeltaU1}) and (\ref{LU}) (applied to $\Phi^{\Downarrow}$), $\Delta^{-1}L^{\Downarrow}\Delta U_i \subseteq U_{i-1}$.
By (\ref{LU}), $L U_i \subseteq U_{i-1}$.
Thus $(\Delta^{-1}L^{\Downarrow}\Delta-L)U_i\subseteq U_{i-1}$.  By these comments, $(\Delta^{-1}A^{*}\Delta- A^{*})U_i\subseteq U_{i-1}$.

To obtain $(\Delta^{-1}A^{*}\Delta- A^{*})U_0=0$, use (\ref{LU}) (applied to both $\Phi$ and $\Phi^{\Downarrow}$) and (\ref{eq:DeltaU1}).
\end{proof}

%%%%%%%%%%%%%%%%%%%%%%%%%%%%%%%%%%%%%%%%
%%%%%%%%%%%%%%%%%%%%%%%%%%%%%%%%%%%%%%%%
\section{More on $\Delta$}
%%%%%%%%%%%%%%%%%%%%%%%%%%%%%%%%%%%%%%%%
%%%%%%%%%%%%%%%%%%%%%%%%%%%%%%%%%%%%%%%%
We continue to discuss the TD system $\Phi$ from Definition \ref{def:TDsys}.  
Recall the decomposition of $V$ given in Corollary \ref{cor:Urefine}.  
In this section, we consider the action of $\Delta$ on each of the summands of this decomposition.

\medskip

\begin{lemma}\label{lemma:proj}
Let $0\leq i\leq d/2$.  For $v\in K_i$ and $i\leq j\leq d-i$, both 
\begin{align}
&F_j^{\Downarrow}\tau_{ij}(A)v=\eta_{ij}(A)v, \label{eq:proj1}\\
&F_j\eta_{ij}(A)v=\tau_{ij}(A)v.\label{eq:proj2}
\end{align}
\end{lemma}
\begin{proof}
We first show (\ref{eq:proj1}).  
First suppose $i=j$.  Use (\ref{eq:F_i1}), Lemma \ref{lemma:K}, and the fact that both $\tau_{ii}$ and $\eta_{ii}$ equal 1.  
Now suppose $i<j$.  
By the comments following (\ref{eta}),
$\tau_{ij}-\eta_{ij}$ has degree at most $j-i-1$ and is therefore in $\text{Span}\{\eta_{ih}\}_{h=i}^{j-1}$.  
From this and Lemma \ref{lemma:tauU_i} (applied to $\Phi^{\Downarrow}$) we find that 
\begin{equation}
\left(\tau_{ij}(A)-\eta_{ij}(A)\right)v\in U_i^{\Downarrow}+U_{i+1}^{\Downarrow}+\cdots + U_{j-1}^{\Downarrow}.\label{eq:proj3}  
\end{equation}
Apply $F_j^{\Downarrow}$ to each side of (\ref{eq:proj3}).  
By Definition \ref{def:F_i} (applied to $\Phi^{\Downarrow}$), $F_j^{\Downarrow}$ applied to the right-hand side of (\ref{eq:proj3}) is zero.  By (\ref{eq:F_i1}) and Lemma \ref{lemma:tauU_i} (applied to $\Phi^{\Downarrow}$), $F^{\Downarrow}_j\eta_{ij}(A)v=\eta_{ij}(A)v$. 
Line (\ref{eq:proj1}) follows from the above comments.\\
Line (\ref{eq:proj2}) is similarly obtained.
\end{proof}

\begin{lemma} \label{lemma:Delta1}
For $0\leq i\leq d/2$ and $i\leq j\leq d-i$, let $\Delta_{ij}$ denote the restriction of $\Delta$ to the subspace $\tau_{ij}(A)K_i$.
Then the following diagram commutes.
\begin{diagram}[grid]
&& K_i &&  \\
&\ldTo(2,2)^{\tau_{ij}(A)} &&  \rdTo(2,2)^{\eta_{ij}(A)}&\\
\tau_{ij}(A)K_i  &&  \rTo^{\Delta_{ij}}   && \eta_{ij}(A)K_i \\
\end{diagram}
\end{lemma}
\begin{proof}
Let $v\in K_i$.  We push $v$ around the diagram.  
Observe that $\Delta_{ij}\tau_{ij}(A)v=\Delta\tau_{ij}(A)v$.  Consider $\Delta\tau_{ij}(A)v$.    In this expression, eliminate $\Delta$ using Definition \ref{def:Delta}.  Then simplify the result using Definition \ref{def:F_i}, Lemma \ref{lemma:tauU_i} (applied to both $\Phi$ and $\Phi^{\Downarrow}$), and Lemma \ref{lemma:proj}.
By these comments we find $\Delta\tau_{ij}(A)v=\eta_{ij}(A)v$.
\end{proof}

\medskip

%%%%%%%%%%%%%%%%%%%%%%%%%%%%%%%%%%%%%%%%%%%%%%%%%%%%%%%%%%
We emphasize a point for later use.  
By Lemma \ref{lemma:Delta1}, we see that for $0\leq i\leq d/2$  and  $i\leq j\leq d-i$,
\begin{equation}
\Delta \tau_{ij}(A)v=\eta_{ij}(A)v \qquad \qquad\qquad (v\in K_i).\label{eq:Delta}
\end{equation}
Setting $j=i$ in the above argument, we see that 
\begin{equation}
(\Delta-I)K_i=0.\label{eq:DeltaK_i}
\end{equation}

%%%%%%%%%%%%%%%%%%%%%%%%%%%%%%%%%%
%%%%%%%%%%%%%%%%%%%%%%%%%%%%%%%%%%
%%%%%%%%%%%%%%%%%%%%%%%%%%%%%%%%%%
\section{The linear transformation $\Psi$}\label{section:psi}
%%%%%%%%%%%%%%%%%%%%%%%%%%%%%%%%%%
%%%%%%%%%%%%%%%%%%%%%%%%%%%%%%%%%%
%%%%%%%%%%%%%%%%%%%%%%%%%%%%%%%%%%
We continue to discuss the TD system $\Phi$ from Definition \ref{def:TDsys}.  
We now introduce a certain linear transformation $\Psi\in {\rm End}(V)$ which has properties that we find attractive.  To define $\Psi$ we give its action on each summand in the decomposition of $V$ from Corollary \ref{cor:Urefine}.  It will turn out that $\Psi$ is the unique linear transformation such that both 
\begin{align*}
\Psi E_iV&\subseteq E_{i-1}V+E_iV+E_{i+1}V,\\
\left(\Psi-\frac{\Delta-I}{\theta_0-\theta_d}\right)E_i^*V&\subseteq E_0^*V+E_1^*V+\cdots + E_{i-2}^*V  
\end{align*}
for $0\leq i\leq d$.  
This characterization of $\Psi$ will be discussed in Section \ref{section:anotherchar}.

\medskip
\begin{lemma}\label{lemma:psidef}
There exists a unique linear transformation $\Psi\in {\rm End}(V)$ such that
\begin{equation}
\Psi \tau_{ij}(A)-\left(\vartheta_j-\vartheta_i\right)\tau_{i,j-1}(A)\label{eq:psidef}
\end{equation}
vanishes on $K_i$ for $0\leq i\leq d/2$ and $i\leq j\leq d-i$.  Recall that $\tau_{i,i-1}=0$.
\end{lemma}
\begin{proof}
By Corollary \ref{cor:Urefine} the sum in (\ref{eq:Urefine1}) is a decomposition of $V$.  In the statement of the lemma, we specified the action of $\Psi$ on each summand and therefore $\Psi$ exists.
The uniqueness assertion is clear.
\end{proof}\\

We clarify the meaning of $\Psi$.  Fix an integer $i$ $(0\leq i\leq d/2)$.  Lemma \ref{lemma:psidef} implies that $\Psi K_i=0$.  More generally, for $i\leq j\leq d-i$ and $v\in K_i$,
\begin{equation}
\Psi \tau_{ij}(A)v = (\vartheta_j-\vartheta_i)\tau_{i,j-1}(A)v.\label{eq:psi-action}
\end{equation}\\  

We look at $\Psi$ from several perspectives.

\medskip

%%%%%%%%%%%%%%%%%%%%%%%%%%%%%%%%%%
\begin{lemma}\label{lemma:psiU}
With reference to Lemma \ref{lemma:psidef},
\begin{equation*}
\Psi U_j  \subseteq U_{j-1} \qquad\qquad (1\leq j\leq d),\qquad \Psi U_0 =0.\label{eq:psiU111}
\end{equation*}
%Moreover, the restriction of $\Psi$ to $U_j$ is an injection if $j> d/2$ and a surjection if $j\leq d/2$ provided that we are in the situation of (i), (ii), or (iv) from Corollary \ref{cor:varthetaj-i}.
\end{lemma}
\begin{proof}
We first show $\Psi U_j\subseteq U_{j-1}$ for $1\leq j\leq d$.
Let $j$ be given.
Recall from (\ref{eq:Udecomp}) the direct sum $U_j=\sum_{i=0}^{\min\{j,d-j\}} \tau_{ij}(A) K_i$.  
Referring to this sum, we will show $\Psi$ sends each summand into $U_{j-1}$.  
Consider the $i^{th}$ summand $\tau_{ij}(A)K_i$.  First suppose $i=j$.  Then $\Psi$ sends this summand to zero because $\Psi K_i=0$.
Next suppose $i<j$.  Using Lemma \ref{lemma:tauU_i} and (\ref{eq:psi-action}), we obtain
\begin{equation*}
\Psi\tau_{ij}(A)K_i\subseteq \tau_{i,j-1}(A)K_i\subseteq U_{j-1}.
\end{equation*}

We now show $\Psi U_0=0$.  Recall that $\Psi K_0=0$.  The result follows since $K_0=U_0$.

%
%By Lemma \ref{lemma:6.5}, $\tau_{ij}(A)v$ is nonzero.
%Since $i<j$, (\ref{eq:psidef3}) implies that $\vartheta_j=\vartheta_i$.
%Lemma \ref{lemma:vartheta=} then implies that $j=i$ or $j=d-i+1$.  
%Thus we reach a contradiction and $\Psi \tau_{ij}(A)v$ is nonzero.
%Hence the restriction of $\Psi$ to $U_j$ is injective.
%
%We now show that the restriction of $\Psi$ to $U_j$ is a surjection for $j\leq d/2$.
%Suppose that $j\leq d/2$.  Let $0\leq i< j$ and let $v\in K_i$.
%Then $\tau_{i,j-1}(A)v \in U_{j-1}$.  We show that there exists $w\in U_j$ such that $\Psi w= \tau_{i,j-1}(A)v$.  
%By Lemma \ref{lemma:vartheta=} we have $\vartheta_j-\vartheta_i\neq 0$.
%It follows from (\ref{eq:psidef3}) that 
%\begin{equation*}
%\Psi \frac{1}{\vartheta_j-\vartheta_i}\tau_{ij}(A)v=\tau_{i,j-1}(A)v.
%\end{equation*}
%Hence
%the restriction of $\Psi$ to $U_j$ is surjective.
\end{proof}

%%%%%%%%%%%%%%%%%%%%%%%%%%%%%%%%%%
\begin{lemma}\label{lemma:psiproj}
With reference to Lemma \ref{lemma:psidef},
\begin{equation}
F_i\Psi=\Psi F_{i+1} \qquad\qquad (0\leq i\leq d-1),\qquad \Psi F_0=0,\qquad F_d\Psi=0.\label{eq:psiproj-1-1}
\end{equation}
\end{lemma}
\begin{proof}
We first show that $F_i\Psi=\Psi F_{i+1}$ for $0\leq i\leq d-1$.  Let $i$ be given.
Recall the decomposition $\{U_j\}_{j=0}^d$ of $V$ from Theorem \ref{thm:U}.  We will show that $F_i\Psi-\Psi F_{i+1}$ vanishes on each $U_j$.   
%Recall from (\ref{eq:F_i1}) that $(F_h-I)U_h=0$ for $0\leq h\leq d$.  
Observe that 
\begin{equation}
F_i\Psi-\Psi F_{i+1}=(F_i-I)\Psi-\Psi (F_{i+1}-I).\label{eq:pp1}
\end{equation}
The right-hand side of (\ref{eq:pp1}) vanishes on $U_j$ by Definition \ref{def:F_i} and Lemma \ref{lemma:psiU}.  Thus $F_i\Psi-\Psi F_{i+1}$ vanishes $U_j$ and hence on $V$.  The equation on the left in (\ref{eq:psiproj-1-1}) follows from the above comments.

The assertions $\Psi F_0=0$, $F_d\Psi=0$ follow from Lemma \ref{lemma:psiU}.
\end{proof}

\medskip

%%%%%%%%%%%%%%%%%%%%%%%%%%%%%%%%%%

\begin{lemma}\label{lemma:DeltaIPsi}
With reference to Definition \ref{def:Delta} and Lemma \ref{lemma:psidef}, 
for $0\leq j\leq d$ apply either of 
\begin{equation}
\Delta-I-(\theta_0-\theta_d)\Psi, \qquad\qquad \Delta^{-1}-I+(\theta_0-\theta_d)\Psi\label{eq:DeltaIPsi2-1}
\end{equation}
to $U_j$ and consider the image.  This image is contained in $U_0+U_1+\cdots + U_{j-2}$ if $j\geq 2$ and equals $0$ if $j<2$.
\end{lemma}
\begin{proof}
We first consider 
the expression on the left in (\ref{eq:DeltaIPsi2-1}).  
Recall the direct sum $U_j=\sum_{i=0}^{\min\{j,d-j\}} \tau_{ij}(A) K_i$ from (\ref{eq:Udecomp}).  Consider a summand $\tau_{ij}(A)K_i$.  We show that the image of $\tau_{ij}(A)K_i$ under the expression on the left in (\ref{eq:DeltaIPsi2-1}) is contained in $U_0+U_1+\cdots + U_{j-2}$ if $j\geq 2$ and equals $0$ if $j<2$.  
By (\ref{eq:Delta}) and Lemma \ref{lemma:psidef}, the actions of the expression on the left in (\ref{eq:DeltaIPsi2-1}) times $\tau_{ij}(A)$ and 
\begin{equation}
\eta_{ij}(A)-\tau_{ij}(A)-(\theta_0-\theta_d)(\vartheta_j-\vartheta_i)\tau_{i,j-1}(A)\label{eq:DeltaIPsi2-2}
\end{equation}
agree on $K_i$.  
By (\ref{tau}), (\ref{eta}), and Definition \ref{def:vartheta},
(\ref{eq:DeltaIPsi2-2}) is a polynomial in $A$ of degree at most $j-i-2$ if $j\geq i+2$ and equals $0$ if $j<i+2$.  The result follows from the above comments and (\ref{eq:A^kU}).  

We now consider the expression on the right in (\ref{eq:DeltaIPsi2-1}).  We will use the fact that the result holds for the expression on the left in (\ref{eq:DeltaIPsi2-1}).
Observe that 
\begin{equation*}
\Delta^{-1}-I+(\theta_0-\theta_d)\Psi=\Delta^{-1}(\Delta-I)^2-\Delta+I+(\theta_0-\theta_d)\Psi.
\end{equation*}
The result follows from the above comments, (\ref{eq:DeltaU2}) and Lemma \ref{lemma:Delta^{-1}U}.
\end{proof}

\begin{lemma}\label{lemma:RPsi}
With reference to Lemma \ref{lemma:psidef},
$\Psi$ satisfies  
\begin{equation}
\Psi R - R\Psi = \sum_{h=0}^{d} \frac{ \theta_h - \theta_{d-h} }{ \theta_0 - \theta_d } F_h. \label{PsiR}
\end{equation}
\end{lemma}
\begin{proof}
Referring to the decomposition of $V$ given in Corollary \ref{cor:Urefine}, consider any summand $\tau_{ij}(A)K_i$.  We apply each side of (\ref{PsiR}) to this summand. We claim that on this summand, each side of (\ref{PsiR}) acts as $(\theta_j - \theta_{d-j})( \theta_0 - \theta_d )^{-1}I$.

The claim holds for the right-hand side of (\ref{PsiR}) by Definition \ref{def:F_i} and the fact that $\tau_{ij}(A)K_i\subseteq U_j$.  
Concerning the left-hand side of (\ref{PsiR}), we routinely carry out this application using (\ref{tau}), (\ref{line:varthetadiff}), Lemma \ref{lemma:RA}, and Lemma \ref{lemma:psidef}. 
\end{proof}

\medskip

\begin{cor}\label{cor:ADeltaPsiR}
With reference to Definition \ref{def:Delta} and Lemma \ref{lemma:psidef}, \begin{equation*}
\frac{A - \Delta^{-1}A\Delta}{\theta_0-\theta_d} = \Psi R-R\Psi. \label{eq:ADeltaRPsi}
\end{equation*}
\end{cor}
\begin{proof}
Use Corollary \ref{cor:ADelta-2} and Lemma \ref{lemma:RPsi}.
\end{proof}
 
 \medskip
 
 We now give a characterization of $\Psi$.
 
\begin{lemma}
Given $\Psi'\in {\rm End}(V)$ such that 
\begin{equation}
\Psi' R - R\Psi' = \sum_{h=0}^{d} \frac{ \theta_h - \theta_{d-h} }{ \theta_0 - \theta_d } F_h\label{eq:psichar1-1-1}
\end{equation}
and 
$\Psi' K_i=0$ for $0\leq i\leq d/2$.  Then $\Psi'=\Psi$.
\end{lemma}
\begin{proof}
Recall from Corollary \ref{cor:Urefine} the decomposition $V=\sum_{i=0}^{r}\sum_{j=i}^{d-i} \tau_{ij}(A)K_i$, where $r=\lfloor d/2\rfloor$.  
We show that $\Psi-\Psi'$ vanishes on each summand by fixing $i$ and inducting on $j$.  
Let $i$ be given.  
Recall that $\Psi K_i=0$.  
Thus $\Psi-\Psi'$ vanishes on $\tau_{ii}(A)K_i=K_i$.  
Now suppose $\Psi-\Psi'$ vanishes on $\tau_{ij}(A)K_i$.  We show that $\Psi-\Psi'$ vanishes on $\tau_{i,j+1}(A)K_i$.  
By (\ref{PsiR}) and (\ref{eq:psichar1-1-1}), we see that 
\begin{equation*}
(\Psi-\Psi') R = R(\Psi -\Psi' ).\label{eq:RPsichar}
\end{equation*}
By the above comments, $\Psi-\Psi'$ vanishes on $R\tau_{ij}(A)K_i$.  
By (\ref{tau}) and Lemma \ref{lemma:RA}, $R\tau_{ij}(A)K_i=\tau_{i,j+1}(A)K_i$.  
Thus $\Psi-\Psi'$ vanishes on $\tau_{i,j+1}(A)K_i$.  
So $\Psi-\Psi'$ vanishes on $V$. 
\end{proof}

\medskip

%%%%%%%%%%%%%%%%%%%%%%%%%%%%%%%

\begin{lemma}\label{lemma:psiU2}
With reference to Definition \ref{def:Delta} and Lemma \ref{lemma:psidef}, 
$\Delta^{-1}A^*\Delta-A^*$ acts on $U_i$ as
\begin{equation*}
(\theta_{i-1}^*-\theta_i^*)(\theta_0-\theta_d)\Psi
\end{equation*} 
for $1\leq i\leq d$ and as $0$ for $i=0$.
\end{lemma}
\begin{proof}
First assume $1\leq i\leq d$.  For notational convenience, we abbreviate $\Omega=(\theta_0-\theta_d)\Psi$.  We will show that 
\begin{equation}
\Delta^{-1}A^*\Delta-A^*-(\theta_{i-1}^*-\theta_{i}^*)\Omega \label{eq:psiU2-31}
\end{equation}
vanishes on $U_i$.
To accomplish this, we show that the image of $U_i$ under (\ref{eq:psiU2-31}) is contained in both $U_{i-1}$ and $\sum_{h=0}^{i-2}U_h$.

We first show that the image of $U_i$ under (\ref{eq:psiU2-31}) is contained in $U_{i-1}$.  
This follows from Lemma \ref{lemma:DeltaA^*U} and Lemma \ref{lemma:psiU}.

We now show that the image of $U_i$ under (\ref{eq:psiU2-31}) is contained in $\sum_{h=0}^{i-2}U_h$.
Observe that (\ref{eq:psiU2-31}) is equal to 
\begin{equation}\label{eq:psiU2-41}
\begin{split}
\theta_{i-1}^*(\Delta^{-1}-I)\Omega \ +\ \Delta^{-1}(A^*-\theta_{i-1}^*I)\Omega \
&+\ (\Delta^{-1}-I)(A^*-\theta_i^*I)\\
&+\ \Delta^{-1}A^*(\Delta-I-\Omega)
\ +\ \theta_i^*(\Delta^{-1}-I+\Omega).
\end{split}
\end{equation}
We will argue that each of the five terms in this sum sends $U_i$ into $\sum_{h=0}^{i-2}U_h$.
We begin by recalling some facts.  
For $0\leq j\leq d$ each of 
\begin{equation*}
A^*-\theta_j^*I, \qquad \Delta-I, \qquad \Delta^{-1} - I,\qquad \Omega
\end{equation*}
sends $U_j$ into $\sum_{h=0}^{j-1}U_h$.  This is a consequence of Theorem \ref{thm:U}(ii), (\ref{eq:DeltaU2}), Lemma \ref{lemma:Delta^{-1}U} and Lemma \ref{lemma:psiU} respectively.
It follows from these comments that for $0\leq j\leq d$, each of $A^*$, $\Delta$, $\Delta^{-1}$, $\Omega$ sends $U_j$ into $\sum_{h=0}^{j}U_h$.  Using the above facts we find that each of 
\begin{equation*}
(\Delta^{-1}-I)\Omega, \qquad \Delta^{-1}(A^*-\theta_{i-1}^*I)\Omega, \qquad (\Delta^{-1}-I)(A^*-\theta_i^*I)
\end{equation*}
sends $U_i$ into $\sum_{h=0}^{i-2}U_h$.  
Thus each of the first three terms in the sum (\ref{eq:psiU2-41}) sends $U_i$ into $\sum_{h=0}^{i-2} U_h$.
By Lemma \ref{lemma:DeltaIPsi}, 
each of 
\begin{equation*}
\Delta-I-\Omega, \qquad \Delta^{-1}-I+\Omega
\end{equation*}
 sends $U_i$ into $\sum_{h=0}^{i-2}U_h$.  
 By the above facts, each of the last two terms in the sum (\ref{eq:psiU2-41}) sends $U_i$ into $\sum_{h=0}^{i-2} U_h$.
We have now shown that each of the five terms in the sum (\ref{eq:psiU2-41}) sends $U_i$ into $\sum_{h=0}^{i-2} U_h$.  Therefore, the image of $U_i$ under (\ref{eq:psiU2-41}) is contained in $\sum_{h=0}^{i-2} U_h$.

By the above comments and Theorem \ref{thm:U}, the expression (\ref{eq:psiU2-31}) vanishes on $U_i$.  The proof is complete for $1\leq i\leq d$.

The case when $i=0$ follows from Lemma \ref{lemma:DeltaA^*U}.
\end{proof}

\medskip

Combining Lemma \ref{lemma:psiU2} with Lemma \ref{lemma:LDelta}, we obtain the following corollary.

\begin{cor}
With reference to Definition \ref{def:Delta} and Lemma \ref{lemma:psidef}, 
$\Delta^{-1}L^{\Downarrow}\Delta-L$ acts on $U_i$ as
\begin{equation*}
(\theta_{i-1}^*-\theta_i^*)(\theta_0-\theta_d)\Psi
\end{equation*} 
for $1\leq i\leq d$ and as $0$ for $i=0$.
\end{cor}

%%%%%%%%%%%%%%%%%%%%%%%%%%%%%%%%%%%%%%%%
%%%%%%%%%%%%%%%%%%%%%%%%%%%%%%%%%%%%%%%%
%%%%%%%%%%%%%%%%%%%%%%%%%%%%%%%%%%%%%%%%
\section{The eigenvalue and dual eigenvalue sequences}
%%%%%%%%%%%%%%%%%%%%%%%%%%%%%%%%%%%%%%%%
%%%%%%%%%%%%%%%%%%%%%%%%%%%%%%%%%%%%%%%%
%%%%%%%%%%%%%%%%%%%%%%%%%%%%%%%%%%%%%%%%
We continue to discuss the TD system $\Phi$ from Definition \ref{def:TDsys}.  
In Sections \ref{section:commute}, \ref{section:anotherchar}, and \ref{section:poly}, we will obtain some detailed results about $\Delta$ and $\Psi$.  
In order to do so, we must first recall some facts concerning the eigenvalues and dual eigenvalues of $\Phi$.  

\medskip

\begin{thm} {\rm \cite[Theorem 11.1]{Somealg}} \label{thm:recurrence}
The expressions
\begin{equation}
\frac{\theta_{i-2}-\theta_{i+1}}{\theta_{i-1}-\theta_i}, \hspace{.75 in} \frac{\theta^*_{i-2}-\theta^*_{i+1}}{\theta^*_{i-1}-\theta^*_i}\label{eq:commonval}
\end{equation}
are equal and independent of $i$ for $2\leq i\leq d-1$.  
\end{thm}

\begin{definition}\label{def:beta}{\rm
We associate a scalar $\beta$ with $\Phi$ as follows.  If $d\geq 3$ let $\beta +1$ denote the common value of (\ref{eq:commonval}).  If $d\leq 2$ let $\beta$ denote any nonzero scalar in $\K$.  We call $\beta$ the {\it base} of $\Phi$.
}\end{definition}

%%%%%%%%%%%%%%%%%%%%%%%%%%%%%%%%%%

\begin{thm}{\rm \cite[Theorem 11.2]{Somealg}}\label{thm:q}
With reference to Definition \ref{def:beta}, the following
{\rm (i)--(iv)} hold.
\begin{enumerate}
\item[{\rm (i)}] 
Suppose $\beta \neq\pm 2$, and pick 
$q \in 
\overline{\K}$ such that 
 $q+q^{-1}=\beta $. Then there exist scalars 
 $\alpha_1, \alpha_2, \alpha_3,
 \alpha^*_1, \alpha^*_2, \alpha^*_3$
 in  
$\overline{\K}$ such that
\begin{eqnarray*}
\theta_i &=& \alpha_1 + \alpha_2q^i + \alpha_3q^{-i},
\\
\theta^*_i &=& \alpha^*_1 + \alpha^*_2q^i + \alpha^*_3q^{-i},
\end{eqnarray*}
for $0 \leq i \leq d$.
Moreover $q^i\not=1$ for $1 \leq i \leq d$.
\item[{\rm (ii)}] Suppose $\beta = 2$ and ${\rm Char}(\K) \not=2$. Then there
exist scalars
 $\alpha_1, \alpha_2, \alpha_3, 
 \alpha^*_1, \alpha^*_2, \alpha^*_3 $  in $\K $ such that
\begin{eqnarray*}
\theta_i &=& \alpha_1 + \alpha_2 i + \alpha_3 i^2,
\\
\theta^*_i &=& \alpha^*_1 + \alpha^*_2 i + \alpha^*_3 i^2, 
\end{eqnarray*}
for $0 \leq i \leq d$.  
Moreover ${\rm Char}(\K)=0$
or ${\rm Char}(\K)>d$.
\item[{\rm (iii)}] Suppose $\beta = -2$ and  ${\rm Char}(\K) \not=2$. Then there
exist scalars
 $\alpha_1, \alpha_2, \alpha_3, 
 \alpha^*_1, \alpha^*_2, \alpha^*_3 $  in $\K $ such that
\begin{eqnarray*}
\theta_i &=& \alpha_1 + \alpha_2 (-1)^i + \alpha_3 i(-1)^i, 
\\
\theta^*_i &=& \alpha^*_1 + \alpha^*_2 (-1)^i + \alpha^*_3 i(-1)^i, 
\end{eqnarray*}
for $0 \leq i \leq d$.  
Moreover ${\rm Char}(\K)=0$
or ${\rm Char}(\K)>d/2$.
\item[{\rm (iv)}] Suppose $\beta = 0$ and  ${\rm Char}(\K) =2$. Then $d= 3$.
\end{enumerate}
\end{thm}

\medskip
%%%%%%%%%%%%%%%%%%%%%%%%%%%%%%%%%%%%%%%%

\begin{lemma}{\rm \cite[Lemma 9.4]{2LT}}\label{lemma:thetaq}
With reference to Definition \ref{def:beta}, pick integers $i, j, r, s$
$(0\leq i,j,r,s\leq d)$ and assume $i+j=r+s$, $i\neq j$.  Then the following {\rm (i)--(iv)} hold.

\begin{enumerate}
\item[{\rm (i)}] Suppose $\beta\neq \pm 2$.  Then 
 \begin{equation*}
 \frac{\theta_r-\theta_s}{\theta_i-\theta_j}=\frac{q^r-q^s}{q^i-q^j},
\end{equation*}
where  $q+q^{-1}=\beta $.
\item[{\rm (ii)}] Suppose $\beta=2$ and ${\rm Char}(\K)\neq 2$.  Then
 \begin{equation*}
 \frac{\theta_r-\theta_s}{\theta_i-\theta_j}=\frac{r-s}{i-j}.
\end{equation*}
\item[{\rm (iii)}] Suppose $\beta=-2$ and ${\rm Char}(\K)\neq 2$.  Then
 \begin{equation*}
 \frac{\theta_r-\theta_s}{\theta_i-\theta_j}= \left\{
     \begin{array}{ll}
       (-1)^{r+i}\frac{r-s}{i-j} & \text{if  }\ i+j \text{ is even},\\
       (-1)^{r+i} & \text{if  }\ i+j \text{ is odd}.
     \end{array}
   \right.
   \end{equation*}
\item[{\rm (iv)}] Suppose $\beta = 0$ and
 ${\rm Char}(\K) =2$. 
 Then
\begin{equation*}
{{\theta_r-\theta_{s}}\over {\theta_i-\theta_j}}
\;=\; 
 \left\{ \begin{array}{ll}
            0  & \mbox{if $\;r=s$}, \\
	1  & \mbox{if $\;r\not=s$.}
				   \end{array}
				\right. 
\end{equation*}
\end{enumerate}
\end{lemma}
\begin{proof}
Use Theorem \ref{thm:q}.
\end{proof}

%%%%%%%%%%%%%%%%%%%%%%%%%%%%%%%%%%%%%%%%
%%%%%%%%%%%%%%%%%%%%%%%%%%%%%%%%%%%%%%%%
%%%%%%%%%%%%%%%%%%%%%%%%%%%%%%%%%%%%%%%%
\section{Some scalars}\label{section:vartheta}
%%%%%%%%%%%%%%%%%%%%%%%%%%%%%%%%%%%%%%%%
%%%%%%%%%%%%%%%%%%%%%%%%%%%%%%%%%%%%%%%%
%%%%%%%%%%%%%%%%%%%%%%%%%%%%%%%%%%%%%%%%
We continue to discuss the TD system $\Phi$ from Definition \ref{def:TDsys}. 
In Section \ref{section:prelim}, we used $\Phi$ to define the scalars $\{\vartheta_i\}_{i=0}^{d+1}$.  In this section we discuss some properties of these scalars which will be of use later.

\medskip

Recall from Definition \ref{def:vartheta} that
\begin{equation*}
\vartheta_i = \sum_{h=0}^{i-1}\frac{\theta_h-\theta_{d-h}}{\theta_0-\theta_d}\qquad\qquad (0\leq i\leq d+1).
\end{equation*}

We remark that 
$$\vartheta_0=0, \qquad \vartheta_1=1,\qquad \vartheta_d=1, \qquad \vartheta_{d+1}=0. $$
Moreover, 
\begin{equation}
\vartheta_i=\vartheta_{d-i+1}\qquad\qquad (0\leq i\leq d+1).\label{line:varthetasym2}
\end{equation}

%%%%%%%%%%%%%%%%%%%%%%%%%%%%%%%%%%

\medskip
\begin{lemma}\label{lemma:varthetasym}
For $0\leq i \leq d$,
\begin{equation*}
\vartheta_{d-i}-\vartheta_i = \frac{\theta_i-\theta_{d-i}}{\theta_0-\theta_d}.
\end{equation*}
\end{lemma}
\begin{proof}
Use (\ref{line:varthetadiff}) and (\ref{line:varthetasym2}).
\end{proof}

\medskip

%%%%%%%%%%%%%%%%%%%%%%%%%%%%%%%%%%
We now express the $\vartheta_i$ in closed form.  

\medskip

\begin{lemma}\label{lemma:vartheta} {\rm \cite[Lemma 10.2]{2LT}}
With reference 
to Definition \ref{def:beta}, the following holds for $0\leq i\leq d+1$.
\begin{enumerate}
\item[{\rm (i)}] Suppose $\beta \neq\pm 2$. Then
\begin{equation*}
\vartheta_i= {{(1-q^i)(1-q^{d-i+1})}\over {(1-q)(1-q^d)}},
\end{equation*}
where  $q+q^{-1}=\beta $.
\item[{\rm (ii)}] Suppose $\beta = 2$ and 
${\rm Char}(\K) \not=2$. Then
\begin{equation*}
 \vartheta_i= {{i(d-i+1)}\over {d}}.
\end{equation*}
\item[{\rm (iii)}] Suppose $\beta = -2$,  
${\rm Char}(\K) \not=2$, and 
$\;d\;$ is odd. Then
\begin{equation*}
\vartheta_i=  \left\{ \begin{array}{ll}
                   0  & \mbox{if $\;i\; $ is even, } \\
				  1 & \mbox{if $\;i\;$ is odd.}
				   \end{array}
				\right. 
\end{equation*}                     
\item[{\rm (iv)}] Suppose $\beta = -2$, 
${\rm Char}(\K) \not=2$, and 
  $\;d\;$ is even. Then 
\begin{equation*}
\vartheta_i = 
  \left\{ \begin{array}{ll}
                   i/d  & \mbox{if $\;i\; $ is even, } \\
	(d-i+1)/d \quad 	   & \mbox{if $\;i\;$ is odd. }
				   \end{array}
				\right.  
\end{equation*}
\item[{\rm (v)}] Suppose $\beta = 0$, 
${\rm Char}(\K) =2$, and 
  $\;d=3$. Then 
\begin{equation*}
\vartheta_i = 
  \left\{ \begin{array}{ll}
                   0  & \mbox{if $\;i\; $ is even, } \\
	1 \quad 	   & \mbox{if $\;i\;$ is odd. }
				   \end{array}
				\right.  
\end{equation*}
\end{enumerate}
\end{lemma}
\begin{proof}
The  above sums
can  be  computed directly using 
Lemma \ref{lemma:thetaq}.
\end{proof}

\medskip

\begin{cor}\label{cor:varthetaneq0}
With reference to Lemma \ref{lemma:vartheta}, 
assume we are in the situation of {\rm (i), (ii)} or {\rm (iv)}.  
Then $\vartheta_i\neq 0$ for $1\leq i\leq d$.
\end{cor}

\medskip

When we were working with the eigenvalues of $\Phi$, a key feature was that they are mutually distinct.  So it is natural to ask if there are any duplications in the sequence $\{\vartheta_i\}_{i=0}^{d+1}$.  
In (\ref{line:varthetasym2}) we already saw that $\vartheta_i=\vartheta_{d-i+1}$ for $0\leq i\leq d+1$.  So we would like to know if  the 
$\{\vartheta_i\}_{i=0}^{r}$ are mutually distinct, where $r=\lfloor \frac{d+1}{2}\rfloor$.
It turns out that 
this is false 
in general, but something can be said in certain cases.  We now explain the details.  

\medskip

%%%%%%%%%%%%%%%%%%%%%%%%%%%%%%%%%%
\begin{cor}\label{cor:varthetaj-i}
With reference 
to Definition \ref{def:beta}, the following holds for $0\leq i,j\leq d+1$.
\begin{enumerate}
\item[{\rm (i)}] Suppose $\beta \neq\pm 2$. Then
\begin{equation*}
\vartheta_i-\vartheta_j= \frac{\left(q^j-q^i\right)\left(1-q^{d-i-j+1}\right)}{\left(1-q\right)\left(1-q^d\right)}.
\end{equation*}
\item[{\rm (ii)}] Suppose $\beta = 2$ and 
${\rm Char}(\K) \not=2$. Then
\begin{equation*}
 \vartheta_i-\vartheta_j= \frac{(i-j)(d-i-j+1)}{d}.
\end{equation*}
\item[{\rm (iii)}] Suppose $\beta = -2$,  
${\rm Char}(\K) \not=2$, and 
$\;d\;$ odd. Then
\begin{equation*}
\vartheta_i-\vartheta_j=  \left\{ \begin{array}{ll}
                   0  & \mbox{if $\;i+j\; $ is even, } \\
				  (-1)^{j} & \mbox{if $\;i+j\;$ is odd.}
				   \end{array}
				\right. 
\end{equation*}                     
\item[{\rm (iv)}] Suppose $\beta = -2$, 
${\rm Char}(\K) \not=2$, and 
  $\;d\;$ even. Then 
\begin{equation*}
\vartheta_i-\vartheta_j = 
  \left\{ \begin{array}{ll}
                  (-1)^{j}\frac{i-j}{d}  & \mbox{if $\;i+j\; $ is even, } \\
	(-1)^{j}\frac{d-i-j+1}{d} \quad 	   & \mbox{if $\;i+j\;$ is odd. }
				   \end{array}
				\right.  
\end{equation*}
\item[{\rm (v)}] Suppose $\beta = 0$, 
${\rm Char}(\K) =2$, and 
  $\;d=3$. Then 
\begin{equation*}
\vartheta_i-\vartheta_j = 
  \left\{ \begin{array}{ll}
                   0  & \mbox{if $\;i+j\; $ is even, } \\
	1 \quad 	   & \mbox{if $\;i+j\;$ is odd. }
				   \end{array}
				\right.  
\end{equation*}
\end{enumerate}
\end{cor}
\begin{proof}
Use Lemma \ref{lemma:vartheta}.
\end{proof}

%%%%%%%%%%%%%%%%%%%%%%%%%%%%%%%%%%

\begin{lemma}\label{lemma:vartheta=}
With reference to Lemma \ref{lemma:vartheta}, 
assume we are in the situation of {\rm (i), (ii)} or {\rm (iv)}.  
Then the following are equivalent for $0\leq i,j\leq d+1$.
\begin{enumerate}
\item[{\rm (i)}] $\vartheta_i=\vartheta_j$.
\item[{\rm (ii)}] $i=j$ or $i+j=d+1$.
\end{enumerate}
\end{lemma}
\begin{proof}
Use Theorem \ref{thm:q} and Corollary \ref{cor:varthetaj-i}.
\end{proof}

\medskip

%%%%%%%%%%%%%%%%%%%%%%%%%%%%%%%%%%

We finish this section with a comment.

\medskip
%%%%%%%%%%%%%%%%%%%%%%%%%%%%%%%%%%

\begin{lemma}\label{lemma:thetavartheta} 
For $0\leq i,j,r,s\leq d$ we have
\begin{equation*}
\left(\theta_r-\theta_s\right)\left(\vartheta_i-\vartheta_j\right)=\left(\theta_i-\theta_j\right)\left(\vartheta_r-\vartheta_s\right),
\end{equation*}
provided that $i+j=r+s$.
\end{lemma}
\begin{proof}
Use Lemma \ref{lemma:thetaq} and Corollary \ref{cor:varthetaj-i}.
\end{proof}

%%%%%%%%%%%%%%%%%%%%%%%%%%%%%%%%%%
%%%%%%%%%%%%%%%%%%%%%%%%%%%%%%%%%%
%%%%%%%%%%%%%%%%%%%%%%%%%%%%%%%%%%
\section{The scalars $[r,s,t]$}
%%%%%%%%%%%%%%%%%%%%%%%%%%%%%%%%%%
%%%%%%%%%%%%%%%%%%%%%%%%%%%%%%%%%%
%%%%%%%%%%%%%%%%%%%%%%%%%%%%%%%%%%
We continue to discuss the TD system $\Phi$ from Definition \ref{def:TDsys}.  To motivate our results in this section, for the moment fix an integer $i$ $(0\leq i\leq d/2)$.  
As we proceed, it will be convenient to express each of $\{\tau_{ij}\}_{j=i}^{d-i}$ as a linear combination of $\{\eta_{ij}\}_{j=i}^{d-i}$.
In order to describe the coefficients, we will use the following notation.\\

For all $a,q\in\overline{\K}$ define
\begin{align}
(a;q)_n=(1-a)(1-aq)\cdots (1-aq^{n-1}), & \qquad \qquad n=0,1,2,\mathellipsis \label{eq:aq_n}
\end{align}
and interpret $(a;q)_0=1$.\\

In \cite{LP24} Terwilliger defined some scalars $[r, s, t]_q \in\K$ for nonnegative integers $r, s, t$ such that $r+s+t \leq d$.  By \cite[Lemma 13.2]{LP24} these scalars are rational functions of the base $\beta$.  
In this paper we are going to drop the subscript $q$ and just write $[r, s, t]$.
For further discussion of these scalars see \cite{DP} and \cite{LP24}.  

%%%%%%%%%%%%%%%%%%%%%%%%%%%%%%%%%%
\medskip
\begin{definition}\label{def:triplebracket}{\rm \cite[Lemma 13.2]{LP24}} {\rm
With reference to Definition \ref{def:beta},
let $r,s,t$ denote nonnegative integers such that $r+s+t\leq d$.
We define $\lbrack r,s,t \rbrack$ as follows.
\begin{itemize}
\item[\rm (i)]
Suppose $\beta\neq\pm 2$.  Then
\begin{eqnarray*}
 \lbrack r,s,t\rbrack= 
\frac{
(q;q)_{r+s}
(q;q)_{r+t}
(q;q)_{s+t}
}
{
(q;q)_{r}
(q;q)_{s}
(q;q)_{t}
(q;q)_{r+s+t}
},
\end{eqnarray*}
where $q+q^{-1}=\beta$.
\item[\rm (ii)] 
Suppose $\beta=2$ and ${\rm Char}(\K)\neq 2$.  Then
\begin{equation*}
\lbrack r,s,t \rbrack= {{(r+s)!\,(r+t)!\,(s+t)!}\over{r!\,s!\,t!\,(r+s+t)!}}.
\end{equation*}
\item[\rm (iii)] 
Suppose $\beta=-2$ and ${\rm Char}(\K)\neq 2$.  If each of $r,s,t$ is odd, then
$\lbrack r,s,t\rbrack =0$. If at least one of $r,s,t$ is even,
then
\begin{equation*}
\lbrack r,s,t \rbrack = 
\frac{\lfloor \frac{r+s}{2}\rfloor ! 
\lfloor \frac{r+t}{2}\rfloor !
\lfloor \frac{
s+t}{2}\rfloor !}{\lfloor \frac{r}{2}\rfloor ! 
\lfloor \frac{s}{2}\rfloor ! 
\lfloor \frac{t}{2}\rfloor !  \lfloor \frac{r+s+t}{2} \rfloor !} .
\end{equation*}
The expression $\lfloor x \rfloor $ denotes the greatest integer less
than or equal to $x$.
\item[\rm (iv)] 
Suppose $\beta=0$, ${\rm Char}(\K)=2$, and $d=3$.
If each of $r,s,t$ equals 1, then $\lbrack r,s,t\rbrack =0$.
If at least one of $r,s,t$ equals 0, then $\lbrack r,s,t \rbrack=1$.
\end{itemize}
}\end{definition}

%%%%%%%%%%%%%%%%%%%%%%%%%%%%%%%%%%
\medskip

We make a few observations.  The expression $[r,s,t]$ is symmetric in $r$, $s$, $t$.
Also, $[r,s,t]=1$ if at least one of $r, s, t$ equals zero.

\medskip
\begin{lemma}{\rm \cite[Lemma 5.3]{DP}}\label{lemma:rstu}
Let $r,s,t,u$ denote nonnegative integers such that $r+s+t+u\leq d$.  Then
\begin{equation*}
[r,s,t+u][t,u,r+s]=[s,u,r+t][r,t,s+u].
\end{equation*}
\end{lemma}

\medskip

%%%%%%%%%%%%%%%%%%%%%%%%%%%%%%%%%%

The following result is a modification of {\rm \cite[Lemma 12.4]{switch}}.

\medskip
\begin{lemma} \label{lemma:tau-to-eta}
Let $0\leq i\leq d/2$ and $i\leq j\leq d-i$.  
Both 
\begin{align}
	\tau_{ij}&=\sum_{h=0}^{j-i} \ [h,j-i-h, d-i-j]\tau_{i,i+h}(\theta_{d-i})\eta_{i,j-h},\label{eq:tau-to-eta}\\
	\eta_{ij}&=\sum_{h=0}^{j-i} \ [h,j-i-h, d-i-j]\eta_{i,i+h}(\theta_{i})\tau_{i,j-h}.\label{eq:eta-to-tau}
\end{align}
\end{lemma}
\begin{proof}
Apply \cite[Lemma 12.4]{switch} to the sequence $\{ \theta_k\}_{k=i}^{d-i}$.
\end{proof}\\

Later in the paper, we will be doing some computations involving the coefficients in (\ref{eq:tau-to-eta}) and (\ref{eq:eta-to-tau}).  The following results will aid in these computations.  

\medskip
%%%%%%%%%%%%%%%%%%%%%%%%%%%%%%%%%%
\begin{cor}\label{cor:rstvartheta}
For $0\leq i\leq d/2$ and $i+1\leq j\leq d-i$,
\begin{equation*}
\left(\theta_0-\theta_d\right)\left(\vartheta_j-\vartheta_i\right)=\left(\theta_{i}-\theta_{d-i}\right)[1,j-i-1,d-i-j].
\end{equation*}
\end{cor}
\begin{proof}
Let $C$ denote the coefficient of $x^{j-i-1}$ on either side of (\ref{eq:tau-to-eta}).  From the left-hand side of (\ref{eq:tau-to-eta}), we see 
\begin{equation}
C=-\sum_{h=i}^{j-1}\theta_{h}.\label{eq:rstvartheta-1}
\end{equation}  
From the right-hand side of (\ref{eq:tau-to-eta}), we see 
\begin{equation}
C=\left(\theta_{d-i}-\theta_i\right)[j-i-1,1,d-i-j]-\sum_{h=i}^{j-1}\theta_{d-h}.\label{eq:rstvartheta-2}
\end{equation}  
Subtract (\ref{eq:rstvartheta-1}) from (\ref{eq:rstvartheta-2}) and invoke the symmetry of $[r,s,t]$ as well as Definition \ref{def:vartheta} to get the result.
\end{proof}

\medskip
%%%%%%%%%%%%%%%%%%%%%%%%%%%%%%%%%%

\begin{lemma} \label{lemma:varthetarst}
For $0\leq i\leq d/2$ and $i+1\leq j\leq d-i$ and $0\leq h\leq j-i-1$, 
\begin{equation}\label{eq:varthetarst2}
\begin{split}
\left(\vartheta_j-\vartheta_i\right)&[h,j-i-h-1,d-i-j+1]\\
&= \left(\vartheta_{j-h}-\vartheta_i\right)[h,j-i-h,d-i-j]
\end{split}
\end{equation}
and
\begin{equation}\label{eq:varthetarst1}
\begin{split}
\left(\vartheta_j-\vartheta_i\right)&[h,j-i-h-1,d-i-j+1]\\
&= \left(\vartheta_{i+h+1}-\vartheta_i\right)[h+1,j-i-h-1,d-i-j].
\end{split}
\end{equation}
\end{lemma}
\begin{proof}
For (\ref{eq:varthetarst2}), 
use Lemma \ref{lemma:rstu} with $r=1$, $s=j-i-h-1$, $t=d-i-j$, $u=h$.  Simplify the result using Corollary \ref{cor:rstvartheta} and the fact that $[r,s,t]$ is symmetric in $r, s, t$.

Line (\ref{eq:varthetarst1}) is similarly obtained.
\end{proof}

%%%%%%%%%%%%%%%%%%%%%%%%%%%%%%%%%%

\begin{cor} \label{cor:prod3}
With reference to Lemma \ref{lemma:vartheta}, assume we are in the situation of {\rm (i), (ii)} or {\rm (iv)}.  
For $0\leq i\leq d/2$ and $i\leq j\leq d-i$ and $0\leq h\leq j-i$,  
\begin{equation}
[h,j-i-h,d-i-j]=\prod_{k=0}^{h-1}\frac{\vartheta_{j-k}-\vartheta_i}{\vartheta_{d-i-k}-\vartheta_i}.\label{eq:prod3-n}
\end{equation}
In (\ref{eq:prod3-n}) the denominators are nonzero by Lemma \ref{lemma:vartheta=}. 
 \end{cor}
\begin{proof}
Assume $h\geq 1$; otherwise both sides of (\ref{eq:prod3-n}) equal $1$.
From (\ref{eq:varthetarst1}) we obtain
\begin{equation*}
[h,j-i-h,d-i-j]=\frac{\vartheta_j-\vartheta_i}{\vartheta_{i+h}-\vartheta_i}[h-1,j-i-h,d-i-j+1].
\end{equation*}
Iterating this we get 
\begin{equation*}
[h,j-i-h,d-i-j]=\prod_{k=0}^{h-1}\frac{\vartheta_{j-k}-\vartheta_i}{\vartheta_{i+k+1}-\vartheta_i}.\end{equation*}
Evaluating the denominator using Lemma \ref{lemma:vartheta=} we obtain the result.
\end{proof}

%%%%%%%%%%%%%%%%%%%%%%%%%%%%%%%%%%
%%%%%%%%%%%%%%%%%%%%%%%%%%%%%%%%%%
%%%%%%%%%%%%%%%%%%%%%%%%%%%%%%%%%%
\section{The maps $\Delta, \Psi$ commute}\label{section:commute}
%%%%%%%%%%%%%%%%%%%%%%%%%%%%%%%%%%
%%%%%%%%%%%%%%%%%%%%%%%%%%%%%%%%%%
%%%%%%%%%%%%%%%%%%%%%%%%%%%%%%%%%%
We continue to discuss the TD system $\Phi$ from Definition \ref{def:TDsys}.  
In Section \ref{section:Delta}, we introduced the linear transformation $\Delta$ and discussed some of its properties.  
In Section \ref{section:psi}, we introduced the linear transformation $\Psi$ and discussed some of its properties.  
We now discuss how $\Delta,\Psi$ relate to each other.
Along this line we have two main results.  They are Theorem \ref{thm:deltapsicomm} and Theorem \ref{thm:Deltapoly}.  We prove Theorem \ref{thm:deltapsicomm} in this section.  Before proving Theorem \ref{thm:Deltapoly}, it will be convenient to give the characterization of $\Psi$ 
discussed in the Introduction.  This will be done in Section \ref{section:anotherchar}.

\medskip

\begin{thm}\label{thm:deltapsicomm}
With reference to Definition \ref{def:Delta} and Lemma \ref{lemma:psidef}, the operators $\Delta$, $\Psi$ commute.
\end{thm}
\begin{proof}
Recall the decomposition of $V$ given in Corollary \ref{cor:Urefine}.  We will show that $\Psi\Delta$, $\Delta\Psi$ agree on each summand $\tau_{ij}(A)K_i$.  

First assume that $i=j$.  Recall that $\tau_{ii}$ and $\eta_{ii}$ both equal $1$.  Using (\ref{eq:DeltaK_i}) and the fact that $\Psi K_i=0$, we routinely find that each of $\Psi\Delta$, $\Delta\Psi$ vanishes on $\tau_{ii}(A)K_i$.

Next assume that $i<j$.  In order to show that $\Psi\Delta$, $\Delta\Psi$ agree on $\tau_{ij}(A)K_i$, it suffices to show that $\Psi\Delta\tau_{ij}(A)$ and $\Delta\Psi\tau_{ij}(A)$ agree on $K_i$.
By (\ref{eq:eta-to-tau}), Lemma \ref{lemma:Delta1},  and Lemma \ref{lemma:psidef}, the operators $\Psi\Delta\tau_{ij}(A)$ and 
\begin{equation}
\sum_{h=0}^{j-i-1} \left(\vartheta_{j-h}-\vartheta_i\right)[h,j-i-h,d-i-j]\eta_{i,i+h}(\theta_{i})\tau_{i,j-h-1}(A)\label{eq:dpc2}
\end{equation}
agree on $K_i$.
By (\ref{eq:eta-to-tau}), Lemma \ref{lemma:Delta1} and Lemma \ref{lemma:psidef}, the operators $\Delta\Psi\tau_{ij}(A)$ and 
\begin{equation}
\left(\vartheta_j-\vartheta_i\right)\sum_{h=0}^{j-i-1}[h,j-i-h-1,d-i-j+1] \eta_{i,i+h}(\theta_i)\tau_{i,j-h-1}(A)
\label{eq:dpc1}
\end{equation}
agree on $K_i$.  In order to show (\ref{eq:dpc2}), (\ref{eq:dpc1}) agree on $K_i$, we will need the fact that
\begin{equation*}
\left(\vartheta_{j-h}-\vartheta_i\right)[h,j-i-h,d-i-j]
\label{eq:dpc4}
\end{equation*}
and
\begin{equation*}
\left(\vartheta_j-\vartheta_i\right)[h,j-i-h-1,d-i-j+1]
\label{eq:dpc3}
\end{equation*}
are equal for $0\leq h\leq j-i-1$.  This equality is (\ref{eq:varthetarst2}).
Therefore (\ref{eq:dpc2}) and (\ref{eq:dpc1}) agree on $K_i$.  
Thus $\Psi\Delta\tau_{ij}(A)$ and $\Delta\Psi\tau_{ij}(A)$ agree on $K_i$.
Hence $\Psi\Delta$, $\Delta\Psi$ agree on $\tau_{ij}(A)K_i$.
By Corollary \ref{cor:Urefine}, $\Psi\Delta$, $\Delta\Psi$ agree on $V$.
\end{proof}\\

%%%%%%%%%%%%%%%%%%%%%%%%%%%%%%%%%%

\medskip

From Theorem \ref{thm:deltapsicomm}, we derive a number of corollaries.

\medskip

\begin{cor}\label{cor:psis}
With reference to Lemma \ref{lemma:psidef}, $\Psi^{\Downarrow}=\Psi$.
\end{cor}
\begin{proof} We first show that $\Psi^{\Downarrow}\Delta=\Delta\Psi$.  Recall the decomposition of $V$ given in Corollary \ref{cor:Urefine}.  We will show that $\Psi^{\Downarrow}\Delta$, $\Delta\Psi$ agree on each summand $\tau_{ij}(A)K_i$.  
By (\ref{eq:Delta}) and (\ref{eq:psi-action}) 
(applied to both $\Phi$ and $\Phi^{\Downarrow}$), $\Psi^{\Downarrow}\Delta\tau_{ij}(A)$ and $\Delta\Psi\tau_{ij}(A)$
agree on $K_i$.  Hence $\Psi^{\Downarrow}\Delta$, $\Delta\Psi$
agree on $\tau_{ij}(A)K_i$.  By Corollary \ref{cor:Urefine}, $\Psi^{\Downarrow}\Delta$, $\Delta\Psi$ agree on $V$.  Thus $\Psi^{\Downarrow}\Delta=\Delta\Psi$.
Combine this fact with Theorem \ref{thm:deltapsicomm} and the fact that $\Delta$ is invertible to get the result.
\end{proof}

%%%%%%%%%%%%%%%%%%%%%%%%%%%%%%%%%%

\begin{cor}\label{cor:psiUdd}
With reference to Lemma \ref{lemma:psidef}, we have
\begin{equation*}
\Psi U_i^{\Downarrow}  \subseteq U_{i-1}^{\Downarrow} \qquad\qquad (1\leq i\leq d), \qquad \Psi U_0^{\Downarrow} =0.
\end{equation*}
\end{cor}
\begin{proof}
Combine Corollary \ref{cor:psis} with Lemma \ref{lemma:psiU}.
\end{proof}

%%%%%%%%%%%%%%%%%%%%%%%%%%%%%%%%%%

\begin{cor}\label{cor:PsiE_iV}  
With reference to Lemma \ref{lemma:psidef}, we have
\begin{align*}
\Psi E_i V  &\subseteq E_{i-1}V+E_iV+E_{i+1}V\quad\qquad (0\leq i\leq d).
\end{align*}
\end{cor}
\begin{proof}
Let $i$ be given.  On the one hand, by Theorem \ref{thm:U}(iii) and Lemma \ref{lemma:psiU}, we have 
\begin{align}
	\Psi E_iV	&\subseteq \Psi (E_iV+E_{i+1}V+\cdots + E_dV) \notag\\
			&= \Psi (U_i+U_{i+1}+\cdots + U_d)\notag\\
			&\subseteq  U_{i-1}+U_{i}+\cdots + U_{d}\notag\\
			&= E_{i-1}V+E_{i+1}V+\cdots + E_dV. \label{eq:PsiEiV1}
\end{align}
On the other hand, by Theorem \ref{thm:U}(iii) applied to $\Phi^{\Downarrow}$ and Corollary \ref{cor:psiUdd}, we have 
\begin{align}
	\Psi E_iV	&\subseteq \Psi (E_0V+E_{1}V+\cdots + E_{i}V)\notag\\
			&= \Psi (U_{d-i}^{\Downarrow}+U_{d-i+1}^{\Downarrow}+\cdots + U_d^{\Downarrow})\notag\\
			&\subseteq  U_{d-i-1}^{\Downarrow}+U_{d-i}^{\Downarrow}+\cdots + U_{d}^{\Downarrow}\notag\\
			&= E_{0}V+E_{1}V+\cdots + E_{i+1}V. \label{eq:PsiEiV2}
\end{align}
Observe that $\Psi E_iV$ is in the intersection of (\ref{eq:PsiEiV1}) and (\ref{eq:PsiEiV2}).  This intersection equals 
$E_{i-1}V+E_i V+E_{i+1}V$, and the result follows.
\end{proof}

%%%%%%%%%%%%%%%%%%%%%%
%%%%%%%%%%%%%%%%%%%%%%
\section{A characterization of $\Psi$}\label{section:anotherchar}
%%%%%%%%%%%%%%%%%%%%%%
%%%%%%%%%%%%%%%%%%%%%%

We continue to discuss the TD system $\Phi$ from Definition \ref{def:TDsys}.
Our goal in this section is to obtain the characterization of $\Psi$ given in the Introduction.

\begin{lemma}
With reference to Lemma \ref{lemma:psidef}, we have
\begin{align*}
\Psi E_i^*V &\subseteq E_0^*V+E_1^*V+\cdots +E_{i-1}^*V\qquad\qquad   (0\leq i\leq d).\label{eq:PsiEi*}
\end{align*}
\end{lemma}
\begin{proof}
Using Theorem \ref{thm:U}(iii) and Lemma \ref{lemma:psiU}, we obtain
\begin{align*}
\Psi E_i^*V &\subseteq \Psi \left( E_0^*V+E_1^*V +\cdots +E_i^*V\right)\\
&= \Psi \left( U_0+U_1+\cdots + U_i \right)\\
&\subseteq U_0+U_1+\cdots + U_{i-1}\\
&=E_0^*V+E_1^*V +\cdots +E_{i-1}^*V.
\end{align*}
\end{proof}

\begin{lemma}\label{lemma:DeltaIPsiEj*}
With reference to Definition \ref{def:Delta} and Lemma \ref{lemma:psidef}, 
for $0\leq j\leq d$ apply either of 
\begin{equation*}
\Delta-I-(\theta_0-\theta_d)\Psi, \qquad\qquad \Delta^{-1}-I+(\theta_0-\theta_d)\Psi
\end{equation*}
to $E_j^*V$ and consider the image.  This image is contained in $E_0^*V+E_1^*V+\cdots + E_{j-2}^*V$ if $j\geq 2$ and equals $0$ if $j<2$.
\end{lemma}
\begin{proof}
Use Theorem \ref{thm:U}(iii) and Lemma \ref{lemma:DeltaIPsi}.
\end{proof}\\

By Corollary \ref{cor:PsiE_iV} and Lemma \ref{lemma:DeltaIPsiEj*}, both
\begin{align*}
\Psi E_iV&\subseteq E_{i-1}V+E_iV+E_{i+1}V,\\
\left(\Psi-\frac{\Delta-I}{\theta_0-\theta_d}\right)E_i^*V&\subseteq E_0^*V+E_1^*V+\cdots + E_{i-2}^*V
\end{align*}
for $0\leq i\leq d$.
We show that these two properties characterize $\Psi$.

\begin{lemma}
Given $\Psi'\in {\rm End}(V)$ such that both
\begin{align*}
\Psi' E_iV&\subseteq E_{i-1}V+E_iV+E_{i+1}V,\\
\left(\Psi'-\frac{\Delta-I}{\theta_0-\theta_d}\right)E_i^*V&\subseteq E_0^*V+E_1^*V+\cdots + E_{i-2}^*V 
\end{align*}
for $0\leq i\leq d$.  Then $\Psi'=\Psi$.
\end{lemma}
\begin{proof}
Recall from Theorem \ref{thm:U} that $\{U_i\}_{i=0}^d$ is a decomposition of $V$.  
It suffices to show that $\Psi,\Psi'$ agree on $U_i$ for $0\leq i\leq d$.  Let $i$ be given.  Observe that 
\begin{equation}
\Psi-\Psi'= \Psi-\frac{\Delta-I}{\theta_0-\theta_d}-\Psi'+\frac{\Delta - I}{\theta_0-\theta_d}.  \label{eq:Psi-Psi'-1}
\end{equation}
 Using (\ref{eq:Psi-Psi'-1}) along with 
Theorem \ref{thm:U}(iii) and Lemma \ref{lemma:DeltaIPsiEj*}, we obtain
\begin{align*}
(\Psi-\Psi')U_i &\subseteq (\Psi-\Psi')(U_0+U_1+\cdots +U_i)\\
&= (\Psi-\Psi')(E_0^*V+E_1^*V+\cdots +E_i^*V)\\
&\subseteq E_0^*V+E_1^*V+\cdots +E_{i-2}^*V\\
&=U_{0}+U_1+\cdots +U_{i-2}.
\end{align*}

By Theorem \ref{thm:U}(iii) and Corollary \ref{cor:PsiE_iV},
\begin{align*}
(\Psi-\Psi')U_i &\subseteq (\Psi-\Psi')(U_i+U_{i+1}+\cdots +U_d)\\
&= (\Psi-\Psi')(E_iV+E_{i+1}V+\cdots +E_dV)\\
&\subseteq E_{i-1}V+E_{i}V+\cdots +E_{d}V\\
&=U_{i-1}+U_i+\cdots +U_d.
\end{align*}

Thus $(\Psi-\Psi')U_i$ is contained in the intersection of $U_{0}+U_1+\cdots +U_{i-2}$ and 
$U_{i-1}+U_i+\cdots +U_d$.  This intersection is zero since $\{U_i\}_{i=0}^d$ is a decomposition of $V$.  So $\Psi-\Psi'$ vanishes on $U_i$.  
Therefore $\Psi,\Psi'$ agree on $U_i$.
\end{proof}

%%%%%%%%%%%%%%%%%%%%%%%%%%%%%%%%%%%%%%%%%
%%%%%%%%%%%%%%%%%%%%%%
%%%%%%%%%%%%%%%%%%%%%%
\section{In general, $\Delta^{\pm 1}$ are polynomials in $\Psi$}\label{section:poly}
%%%%%%%%%%%%%%%%%%%%%%
%%%%%%%%%%%%%%%%%%%%%%

We continue to discuss the TD system $\Phi$ from Definition \ref{def:TDsys}.  
Recall the map $\Delta$ from Definition \ref{def:Delta} and the map $\Psi$ from Lemma \ref{lemma:psidef}.  
In Section \ref{section:commute}, we saw that $\Delta,\Psi$ commute.  In this section, we show that $\Delta^{\pm 1}$ are polynomials in $\Psi$ provided that each of $\vartheta_1,\vartheta_2,\ldots,\vartheta_d$ is nonzero.

\medskip

\begin{thm}\label{thm:Deltapoly}
Let $\Delta\in {\rm End}(V)$ be as in Definition \ref{def:Delta} and
let $\Psi\in {\rm End}(V)$ be as in Lemma \ref{lemma:psidef}.  
With reference to Lemma \ref{lemma:vartheta}, assume we are in the situation of {\rm (i), (ii),} or {\rm (iv)} so that the scalars $\{\vartheta_i\}_{i=1}^d$ from Definition \ref{def:vartheta} are nonzero.  
Then both
\begin{eqnarray}
\Delta = I +\frac{\eta_{1}(\theta_0)}{\vartheta_1}\Psi + \frac{\eta_{2}(\theta_0)}{\vartheta_1\vartheta_2}\Psi^2 +\cdots + \frac{\eta_{d}(\theta_0)}{\vartheta_1\vartheta_2 \cdots \vartheta_d} \Psi^d,\label{eq:MT0}\\
\Delta^{-1} = I +\frac{\tau_{1}(\theta_d)}{\vartheta_1}\Psi + \frac{\tau_{2}(\theta_d)}{\vartheta_1\vartheta_2}\Psi^2  +\cdots+ \frac{\tau_{d}(\theta_d)}{\vartheta_1\vartheta_2 \cdots \vartheta_d} \Psi^d.\label{eq:MT00}
\end{eqnarray}
\end{thm}
\begin{proof}
We first show (\ref{eq:MT0}).  
Recall the decomposition of $V$ from Corollary \ref{cor:Urefine}.  
We show that each side of (\ref{eq:MT0}) agrees on each summand $\tau_{ij}(A)K_i$.
Let $v\in K_i$.
We apply each side of (\ref{eq:MT0}) to the vector $\tau_{ij}(A)v$ and show that the results agree.

We first apply the left-hand side of (\ref{eq:MT0}) to $\tau_{ij}(A)v$.  
By Lemma \ref{lemma:Delta1} and (\ref{eq:eta-to-tau}), 
$\Delta\tau_{ij}(A)v$ is a linear combination of $\{\tau_{i,j-h}(A)v\}_{h=0}^{j-i}$ such that the coefficient of $\tau_{i,j-h}(A)v$ is
\begin{equation}
[h,j-i-h,d-i-j]\eta_{i,i+h}(\theta_i)\label{eq:MT0--1}
\end{equation}
for $0\leq h\leq j-i$.  
We now apply the right-hand side of (\ref{eq:MT0}) to $\tau_{ij}(A)v$.
For the sum on the right-hand side of (\ref{eq:MT0}), the action of each term on $\tau_{ij}(A)v$ is computed using (\ref{eq:psi-action}).  From this computation, one finds that the right-hand side of (\ref{eq:MT0}) applied to $\tau_{ij}(A)v$ is a linear combination of $\{\tau_{i,j-h}(A)v\}_{h=0}^{j-i}$ such that the coefficient of $\tau_{i,j-h}(A)v$ is
\begin{equation}
\frac{\eta_{h}(\theta_0)}{\vartheta_1\vartheta_2\cdots\vartheta_h}\prod_{k=0}^{h-1}\left(\vartheta_{j-k}-\vartheta_i\right)\label{eq:MT0--2}
\end{equation}
for $0\leq h\leq j-i$.  
It remains to show that (\ref{eq:MT0--1}) is equal to (\ref{eq:MT0--2}) for $0\leq h\leq j-i$.  Let $h$ be given.  
By (\ref{eta}) and Corollary \ref{cor:prod3}, the scalar (\ref{eq:MT0--1}) is equal to
\begin{equation}
\prod_{k=0}^{h-1}\frac{\left(\theta_i-\theta_{d-i-k}\right)\left(\vartheta_{j-k}-\vartheta_i\right)}{\vartheta_{d-i-k} -\vartheta_i}.\label{eq:MT0--3}
\end{equation}

By (\ref{eq:eta_j}) and since $\vartheta_{\ell}=\vartheta_{d-\ell+1}$ for $1\leq \ell\leq h$, the scalar (\ref{eq:MT0--2}) is equal to 
\begin{equation}
\prod_{k=0}^{h-1} \frac{\left(\theta_0-\theta_{d-k}\right)\left(\vartheta_{j-k}-\vartheta_i\right)}{\vartheta_{d-k}}. \label{eq:MT0--4}
\end{equation}
By Lemma \ref{lemma:thetavartheta} and since $\vartheta_0=0$,
\begin{align*}
\frac{\theta_i-\theta_{d-i-k}}{\vartheta_{d-i-k}-\vartheta_i}=\frac{\theta_0-\theta_{d-k}}{\vartheta_{d-k}} & \qquad \qquad (0\leq k\leq h-1).
\end{align*}
Using this we find that (\ref{eq:MT0--3}) is equal to (\ref{eq:MT0--4}).  Therefore  
(\ref{eq:MT0--1}) is equal to (\ref{eq:MT0--2}) for $0\leq h\leq j-i$ as desired.  We have shown (\ref{eq:MT0}).

To get (\ref{eq:MT00}), apply (\ref{eq:MT0}) to $\Phi^{\Downarrow}$ and use Corollary \ref{cor:psis} along with the fact that $\vartheta_k^{\Downarrow}=\vartheta_k$ for $1\leq k\leq d$.
\end{proof}

\section{Comments}

We now make a few comments regarding future work related to this paper.\\

The reader may have already noticed that the relation in Lemma \ref{lemma:RPsi} looks like one of the defining relations for the quantum $sl_2$.  In fact, there exists a quantum $sl_2$-module structure here.  
We will treat this topic comprehensively in a future paper.\\

The reader may have also noticed some similarities between $\Delta$ and the switching element $S$ from \cite{switch}.  
In spite of the superficial similarities, we see no connection between $\Delta$ and $S$.\\

We now give some suggestions for further research relating to this paper.

\begin{prob}
With reference to Definition \ref{def:R} and Lemma \ref{lemma:psidef}, what is $L\Psi - \Psi L$?
\end{prob}

\begin{prob}
With reference to Definition \ref{def:R} and Lemma \ref{lemma:psidef}, are $L$ and $\Psi$ related in an interesting way?  How about $L^{\Downarrow}$ and $\Psi$?
\end{prob}

\begin{prob}
With reference to Definition \ref{def:Delta} and Lemma \ref{lemma:psidef}, write $\Psi$ as a polynomial in $\Delta-I$.
\end{prob}

\section{Acknowledgment}

This paper was written while the author was a graduate student at the University of Wisconsin-Madison. The author would like to thank her advisor, Paul Terwilliger, for offering many valuable ideas and suggestions.
%%%%%%%%%%%%%%%%%%%%%%%%%%%%%%%%%%%%%%%%%%%%%%%%%%%%%%%%%%

\bigskip

\noindent Sarah Bockting-Conrad \hfil\break
\noindent Department of Mathematics \hfil\break
\noindent University of Wisconsin \hfil\break
\noindent 480 Lincoln Drive \hfil\break
\noindent Madison, WI 53706-1388 USA \hfil\break
\noindent email: {\tt bockting@math.wisc.edu }\hfil\break
\end{document}